\documentclass{amsart}

\usepackage{amssymb}

\usepackage{comment}

\usepackage{stmaryrd}
\usepackage{multirow}
\usepackage[colorinlistoftodos]{todonotes}

\newtheorem{theorem}{Theorem}[section]
\newtheorem{lem}[theorem]{Lemma}
\newtheorem{prop}[theorem]{Proposition}
\newtheorem{coro}[theorem]{Corollary}

\newcounter{NoMain}

\newtheorem{mainthm}[NoMain]{Theorem}

\theoremstyle{definition}
\newtheorem{definition}[theorem]{Definition}

\newtheorem{Conj}[theorem]{Conjecture}

\theoremstyle{remark}
\newtheorem{remark}[theorem]{Remark}

\usepackage[bookmarks=true]{hyperref}

\numberwithin{equation}{section}

\newcommand{\R}{\mathbb{R}}
\newcommand{\N}{\mathbb{N}}
\newcommand{\Z}{\mathbb{Z}}

\newcommand{\C}{\mathbb{C}}

\newcommand{\fonction}[5]{#1: \begin{array}{ccc}
#2 & \rightarrow & #3 \\
 #4 & \longmapsto & #5 \end{array}}

 \newcommand{\fon}[4]{\begin{array}{ccc}
#1 & \rightarrow & #2 \\
#3 & \longmapsto & #4 \end{array}}

\DeclareMathOperator{\Ima}{Im}
\DeclareMathOperator{\ad}{ad}
\DeclareMathOperator{\Ad}{Ad}
\DeclareMathOperator{\vspan}{span}
\DeclareMathOperator{\Id}{Id}
\DeclareMathOperator{\rank}{rank}

\DeclareMathOperator{\Stab}{Stab}

\DeclareMathOperator{\diam}{diam}
\DeclareMathOperator{\re}{Re}
\DeclareMathOperator{\Hess}{Hess}

\begin{document}

\title{Regularity of $K$-finite matrix coefficients of semisimple Lie groups}

\author{Guillaume Dumas}
\address{Université Claude Bernard Lyon 1, ICJ UMR5208, CNRS, Ecole Centrale de Lyon, INSA Lyon, Université Jean Monnet,
69622 Villeurbanne, France.}

\date{\today}

\email{gdumas@math.univ-lyon1.fr}
\thanks{}

\begin{abstract}
We consider $G$ a semisimple Lie group with finite center and $K$ a maximal compact subgroup of $G$. We study the regularity of $K$-finite matrix coefficients of unitary representations of $G$. More precisely, we find the optimal value $\kappa(G)$ such that all such coefficients are $\kappa(G)$-Hölder continuous. The proof relies on analysis of spherical functions of the symmetric Gelfand pair $(G,K)$, using stationary phase estimates from Duistermaat, Kolk and Varadarajan. If $U$ is a compact form of $G$, then $(U,K)$ is a compact symmetric pair. Using the same tools, we study the regularity of $K$-finite coefficients of unitary representations of $U$, improving on previous results obtained by the author.
\end{abstract}
\maketitle

\section{Introduction}
The aim of this article is to investigate the regularity of (special classes) of matrix coefficients of unitary representations of a semisimple Lie group $G$ with finite center. If $K$ is a maximal compact subgroup of $G$, it is known by the work of Harish-Chandra (\cite{harishchandra}) that matrix coefficients associated to $K$-finite vectors of irreducible unitary representations of $G$ (and more generally admissible representations) are $C^\infty$. Thus, a natural space of matrix coefficients to consider is the space of $K$-finite matrix coefficients. Every unitary representation of $G$ decomposes as a direct integral of irreducible representations, but since Harish-Chandra's estimates depend on the representations, it does not provide any estimates for arbitrary representations.

It turns out that in this case, $(G,K)$ is a Gelfand pair. In this setting, any $K$-bi-invariant coefficient of a unitary representation of $G$ decomposes as a direct integral of positive-definite spherical functions of the pair $(G,K)$ (see Section \ref{sec:gelfand}). Thus, it is a good idea to further restrict the question to such coefficients. The quotient space $G/K$ also has a natural structure of Riemannian symmetric space (of non compact type). In this framework, a lot is known on the spherical functions of the pair (\cite{helgason1979differential},\cite{helgason2000groups}).

Every such symmetric space $G/K$ has a simply connected dual of compact type which is of the form $U/K$, where $U$ is a compact semisimple Lie group containing $K$ as a subgroup. Furthermore, $(U,K)$ is also a Gelfand pair, which we call a compact symmetric pair. Thus, $K$-bi-invariant coefficients of unitary representations of $U$ can also be studied using spherical functions of the pair $(U,K)$. This idea was used by Lafforgue to show that $SO(2)$-bi-invariant coefficients of $SO(3)$ are $\frac{1}{2}$-Hölder outside of singular points, which was a key ingredient in the proof of his strengthening of property (T) for $SL(3,\R)$ (\cite{lafforgue}). In a previous paper (\cite{dumas2023regularity}), we studied these questions for such compact pairs.\medskip

The goal of the paper is to find the optimal value $(r,\delta)\in \N\times [0,1]$ such that any $K$-finite matrix coefficient of $G$ is in the Hölder space $C^{(r,\delta)}(G_r)$ (see Section \ref{sec:holder} for the precise definition of these Hölder spaces). Here, $G_r$ is the dense open subset of regular points of $G$ (see Proposition \ref{prop:KAKv1} and after for the definition). It is important to point out that no regularity result can be proven outside of this space (see Remark \ref{rmk:singular}). However, the set of regular points was sufficient in applications such as in \cite{lafforgue}.

Given a Cartan involution of $\mathfrak{g}$, we have a decomposition  $\mathfrak{g}=\mathfrak{k}\oplus\mathfrak{p}$. If $\mathfrak{a}$ is a maximal abelian subspace of $\mathfrak{p}$, we can consider the associated root system $\Sigma\subset \mathfrak{a}^*$ and a choice of positive roots $\Sigma^+$ (see Section \ref{sec:sphericalfunctions} for more details, in particular for definitions of those objects). 
For $\lambda\in \mathfrak{a}^*$, define \begin{equation*}n(\lambda)=\sum_{\begin{subarray}{c}
    \alpha\in \Sigma^+\\
    \langle \alpha,\lambda\rangle\neq 0
\end{subarray}} m(\alpha)\end{equation*}and set \begin{equation*}
\kappa(G)=\underset{\lambda \in  \mathfrak{a}^*\setminus \{0\}}{\inf} \frac{n(\lambda)}{2}.\end{equation*}

Our main result is the following:
\begin{mainthm}\label{mainthmA}
    Let $G$ be a semisimple Lie group with finite center and $K$ a maximal compact subgroup of $G$. Let $r=\lfloor \kappa(G)\rfloor$ and $\delta=\kappa(G)-r$. Then any $K$-finite matrix coefficient of a unitary representation of $G$ is in $C^{(r,\delta)}(G_r)$. Furthermore, for any $\delta'>\delta$, there exists a $K$-bi-invariant coefficient of a unitary representation of $G$ which is not in $C^{(r,\delta')}(G_r)$.
\end{mainthm}
As explained above, the idea is to first prove this result for $K$-bi-invariant coefficients only. Then, we use the theory of decomposition of a representation into a direct integral to reduce this problem to a question of boundedness of positive-definite spherical functions in Hölder spaces (Lemmas \ref{lem:decinteg} and \ref{lem:lienspheriquekbiinv}). In the setting of semisimple Lie groups, these spherical functions have a well-known expression as an oscillatory integral over the maximal compact subgroup $K$ (\cite[Ch. IV]{helgason2000groups}). The key input is then the use of the stationary phase approximate to obtain estimates on these integrals, using the results from \cite{DKV}. The last step is to prove that the regularity obtained for $K$-bi-invariant coefficient can be enhanced to all $K$-finite coefficients. The arguments for this step follow the same idea as in \cite[Section 5]{dumas2023regularity} for compact semisimple Lie groups, but some steps are more involved due to the lack of a rich finite dimensional representation theory (Section \ref{sec:kfin}).

If $G_\C$ is the simply connected Lie group with Lie algebra $\mathfrak{g}_\C$, and $U$ the analytic subgroup corresponding to the subalgebra $\mathfrak{u}=\mathfrak{k}+i\mathfrak{p}$, it is known that $U$ is a compact simply connected semisimple Lie group and $K\subset U$. Then $(U,K)$ is a Gelfand pair and $U/K$ is a symmetric space of compact type, which is the simply connected compact dual of $G/K$. The regularity of $K$-finite matrix coefficients of unitary representations of $U$ was studied in \cite{dumas2023regularity}. The dual version of Theorem \ref{mainthmA} was obtained only in special cases, namely when $U/K$ is of rank $1$ and when $U/K$ is a Lie group (\cite[Theorems A and B]{dumas2023regularity}). However, given these results and some partial results, it was conjectured that the optimal regularity of $K$-finite coefficients on the subset of regular points $U_r$ should also be $\kappa(G)$ (see \cite[Conjecture 1.3]{dumas2023regularity} or Conjecture \ref{conj} for the statement). With the tools of stationary phase approximation developed for the noncompact case and some complex analysis, we generalise these results. More precisely, we show the following theorem.

\begin{mainthm}\label{mainthmB}
Let $G$ be a semisimple Lie group with finite center and $(U,K)$ defined as above. Let $r=\kappa(G)$ and $\delta=\kappa(G)-r$. Then, there exists an open subset $V\subset U_r$ such that any $K$-finite matrix coefficient of a unitary representation of $U$ is in $C^{(r,\delta)}(V)$. Furthermore, for any $\delta'>\delta$, there exists a $K$-bi-invariant matrix coefficient which is not in $C^{(r,\delta')}(U_r)$.
\end{mainthm}
As explained above, the idea is to use analytic continuation to obtain an expression of spherical functions as an oscillatory integral (Lemma \ref{lem:spheriqueext}) which allow to use the stationary phase approximation. Since the phase function is now complex-valued, the proof is more involved. Theorem \ref{mainthmB} cannot be extended to all of $U_r$, because we only have such an oscillatory integral in a neighbourhood of identity. However, this still shows that the conjecture (\cite[Conjecture 1.3]{dumas2023regularity}) is true at least in some open subset. Furthermore, given some other partial results obtained in \cite{dumas2023regularity}, we also get the full conjecture for several new families of groups (Corollary \ref{coro:cpt2}).

The ideas for the compact case originate from \cite{clerc}. In this paper, Clerc actually gives an expression as an oscillatory integral for any point of $U$. However, the domain of integration lacks compactness. Thus, he could only get estimates for some specific subfamily of positive-definite functions, which is not enough for our purpose. 

\subsection*{Organisation of the paper}
Section \ref{sec:preli} contains preliminaries on Gelfand pairs and spherical functions, especially in the context of semisimple Lie groups. In Section \ref{sec:kfinite}, we reduce the main question from $K$-finite coefficients to $K$-bi-invariant (Theorem \ref{thm:kfinite}) and we explain how to work at the level of the Lie algebra of $G$. Section \ref{sec:mainsec} is devoted to the proof of Theorem \ref{mainthmA} in the case of $K$-bi-invariant coefficients. Finally in Section \ref{sec:cpt}, we study the same questions for compact symmetric pairs and prove Theorem \ref{mainthmB}.

\subsection*{Acknowledgements} I would like to thank my Ph.D. advisor Mikael de la Salle for his involvement. I am thankful to Marco Mazzucchelli for his help. I am indebted to Yannick Guedes Bonthonneau for pointing me out the right statements regarding stationnary phase and listening to my questions.

\section{Preliminaries}\label{sec:preli}
\subsection{Hölder spaces}\label{sec:holder}

\begin{definition}\label{def:holdermultivar}Let $(X,d)$ be a metric space and $U$ open subset of $X$, $(E,\Vert.\Vert)$ a normed vector space, $\alpha\in ]0,1]$. A function $f:U\to E$ is $\alpha$-Hölder if for any compact subset $K$ of $U$, there is $C_K>0$ such that $\forall x,y\in K$, $\Vert f(x)-f(y)\Vert \leq C_Kd(x,y)^\alpha$.

If $X$ is also a normed vector space and $r\in \N$, we say that the map $f$ belongs to $C^{(r,\alpha)}(U,E)$ if $f\in C^r(U,E)$ and the $r$-th differential $D^rf$ is $\alpha$-Hölder as a map from $U$ to the vector space of multilinear $r$-forms. We extend to $\alpha=0$ by $C^{(r,0)}(U,E)=C^r(U,E)$.

For $K$ a compact subset of $U$ and $f\in C^{(r,\alpha)}(U,E)$, define $$\Vert f\Vert_{C^{(r,\alpha)}(K,E)}=\max \left\{\underset{k\leq r}{\max}\, \underset{x\in K}{\sup}\Vert D^kf(x)\Vert,\underset{x,y\in K,x\neq y}{\sup}\frac{\Vert D^rf(x)-D^rf(y)\Vert}{d(x,y)^\alpha}\right\}.$$
The family of semi-norms $\Vert.\Vert_{C^{(r,\alpha)}(K,E)}$ for $K$ a compact subset of $U$ makes the space $C^{(r,\alpha)}(U,E)$ into a Fréchet space.

Finally if $(X,d)$ is a Riemannian manifold, we say that $f\in C^{(r,\alpha)}(U,E)$ if for any chart $(\varphi,V)$ of $U$, $f\circ \varphi^{-1}\in C^{(r,\alpha)}(\varphi(V),E)$.
\end{definition}
\begin{remark}\label{rem:holder}
    If $U$ is locally compact, a function $f:U\to E$ is $\alpha$-Hölder if and only if for any $x\in U$, there exists a neighbourhood $U_x$ of $x$ and a constant $C_x>0$ such that for any $y,z\in U_x$, $\Vert f(y)-f(z)\Vert\leq C_xd(y,z)^\alpha$.
\end{remark}
We will denote $C^{(r,\alpha)}(U,\C)$ by $C^{(r,\alpha)}(U)$.

The following lemma will be useful throughout the article and can be found in \cite[Lemma 2.1]{dumas2023regularity}
\begin{lem}
\label{lem:precomposition}Let $(X,d)$ and $(Y,d')$ be two Riemannian manifolds and $U,V$ open subsets of $X,Y$ respectively. Let $\alpha>0$ and $r\in \N$. Let $\varphi:U\to V$ be a function of class $C^{\infty}$. Then $\varphi_*:f\mapsto f\circ \varphi$ maps $C^{(r,\alpha)}(V)$ to $C^{(r,\alpha)}(U)$ and is continuous.
\end{lem}

\subsection{Gelfand pairs}\label{sec:gelfand}
\begin{definition}Let $G$ be a locally compact topological group with a left Haar measure $dg$ and $K$ a compact subgroup with normalized Haar measure $dk$. The pair $(G,K)$ is a Gelfand pair if the algebra of continuous $K$-bi-invariant functions on $G$ with compact support is commutative for the convolution.

A spherical function of $(G,K)$ is a continuous $K$-bi-invariant non-zero function on $G$ such that for all $x,y\in G$, $$\int_K \varphi(xky)\,dk=\varphi(x)\varphi(y).$$
\end{definition}

A standard result (see \cite[Coro. 6.3.3]{Dijk+2009}) gives a link between spherical functions of $(G,K)$ and unitary representations of $G$.
\begin{prop}
If $(G,K)$ is a Gelfand pair, then for any irreducible unitary representation $\pi$ of $G$ on a Hilbert space $\mathcal{H}$, the subspace $\mathcal{H}^K$ of $K$-invariant vectors is of dimension at most $1$.

The positive-definite spherical functions of $G$ are exactly the matrix coefficients $g\mapsto \langle\pi(g)v,v\rangle$ with $\pi$ an irreducible unitary representation of $G$ and $v$ a $K$-invariant unit vector.

If $G$ is compact, any spherical function is positive-definite.
\end{prop}

More details on Gelfand pairs can be found in \cite[Ch. 5,6,7]{Dijk+2009}.\smallskip

Given a Gelfand pair $(G,K)$, it is natural to study spherical functions in order to get results on $K$-bi-invariant matrix coefficients of unitary representations. Indeed, any matrix coefficient of a unitary representation decomposes into an integral of spherical functions - an infinite sum if $G$ is compact. Then studying boundedness of positive-definite spherical functions in some Hölder spaces is enough to obtain regularity for all $K$-bi-invariant matrix coefficients of unitary representations. More precisely, the optimal regularity of such coefficients is exactly the optimal uniform regularity of spherical functions. The proof of the following two lemmas can be found in \cite[Section 2.2]{dumas2023regularity}.
\begin{lem}\label{lem:decinteg}
Let $(G,K)$ be a Gelfand pair with $G$ second countable. Let $\varphi$ be a $K$-bi-invariant matrix coefficient of a unitary representation $\pi$ on an Hilbert space $\mathcal{H}$. Then, there exists a standard Borel space $X$ and a $\sigma$-finite measure $\mu$ on $X$ such that $$\varphi=\int_X c_x\varphi_x d\mu(x)$$where $\varphi_x$ is a positive-definite spherical function of $(G,K)$ for any $x\in X$ and $c\in L^1(X,\mu)$.
\end{lem}

\begin{lem}\label{lem:lienspheriquekbiinv}Let $(G,K)$ be a Gelfand pair with $G$ a Lie group endowed with a Riemannian metric $d$ and $U$ any open subset of $G$. Let $(\varphi_\lambda)_{\lambda\in \Lambda}$ be the the family of positive-definite spherical functions of $(G,K)$. Then $(\varphi_\lambda)_{\lambda\in \Lambda}$ is bounded in $C^{(r,\delta)}(U)$ if and only if any $K$-bi-invariant matrix coefficient of a unitary representation of $G$ is in $C^{(r,\delta)}(U)$.
\end{lem}

\subsection{Spherical functions of semisimple Lie groups}\label{sec:sphericalfunctions}
Let $G$ be a connected real semisimple Lie group with finite center and $\mathfrak{g}$ its Lie algebra. Note that all results will apply to complex semisimple Lie groups, by viewing them naturally as real Lie groups. Let $\theta$ be a Cartan involution of $\mathfrak{g}$ and $\mathfrak{g}=\mathfrak{k}\oplus \mathfrak{p}$ be the decomposition of $\mathfrak{g}$ in $\pm 1$-eigenspaces of $\theta$. Then $K=\exp \mathfrak{k}$ is a maximal compact subgroup of $G$. Since all maximal compact subgroups of $G$ are conjugated, whenever we say "a maximal compact subgroup of $G$" in the sequel, we may always assume that it is defined by a Cartan involution. Consider $\mathfrak{a}$ a maximal abelian subspace of $\mathfrak{p}$. The rank of $G$ is $\rank G=\dim \mathfrak{a}=\ell$. For $\alpha\in \mathfrak{a}^*$, define $\mathfrak{g}^\alpha=\{X\in \mathfrak{g}\vert \forall H\in \mathfrak{a}, [H,X]=\alpha(H)X\}$ the root space associated to $\alpha$. Let $m(\alpha)=\dim(\mathfrak{g}^\alpha)$ and $\Sigma=\{\alpha\neq 0 \vert m(\alpha)\geq 1\}$ be the set of roots. We say that $\Sigma$ is the restricted root system of $G$. Let $\mathfrak{m}=\mathfrak{g}^0\cap \mathfrak{k}$. Then the Lie algebra decomposes as $$\mathfrak{g}=\mathfrak{m}\oplus \mathfrak{a}\oplus \bigoplus_{\alpha\in \Sigma} \mathfrak{g}^\alpha.$$

The Killing form of $\mathfrak{g}$ induces an inner product on $\mathfrak{a}$, denoted $\langle \cdot,\cdot\rangle$. Then for $\lambda\in \mathfrak{a}^*$, there is a unique $H_\lambda\in \mathfrak{a}$ such that for any $H\in \mathfrak{a}$, $\lambda(H)=\langle H_\lambda,H\rangle$. We use the isomorphism $\lambda \mapsto H_\lambda$ to define an inner product on $\mathfrak{a}^*$ by $$\langle \lambda,\mu\rangle=\langle H_\lambda,H_\mu\rangle.$$

Let $W$ be the Weyl group of the root system $\Sigma$, which is the subgroup of $O(\mathfrak{a}^*)$ generated by the reflections $s_\alpha:x\mapsto x-\frac{2\langle x,\alpha\rangle}{\langle \alpha,\alpha\rangle}\alpha$. The group $W$ also acts on $\mathfrak{a}$ by $wH_{\lambda}=H_{w\lambda}$. By \cite[Thm. 4.3.24]{varadarajan2013lie}, this action can be extended to automorphisms of the Lie algebra $\mathfrak{g}$. In particular, we get that $\mathfrak{g}^{w\alpha}=w(\mathfrak{g}^\alpha)$ and so $m(w\alpha)=m(\alpha)$. The hyperplanes $\{\alpha(H)=0\}$ divide $\mathfrak{a}$ into $\vert W\vert$ connected components. We choose one, which we denote $\mathfrak{a}^+$ and call the positive Weyl chamber, and we define the positive roots $\Sigma^+=\{\alpha \in \Sigma \vert \forall H\in \mathfrak{a}^+,\alpha(H)>0\}$. Then $\Sigma=\Sigma^+ \cup (-\Sigma^+)$. We say that $\alpha\in \Sigma^+$ is simple if it cannot be decomposed as $\alpha=\beta+\gamma$ with $\beta,\gamma\in \Sigma^+$. Let $\Delta$ be the set of simple roots. Then $\Delta$ is a basis of $\mathfrak{a}^*$ and we can write $\Delta=\{\alpha_1,\cdots,\alpha_\ell\}$. Given $\alpha\in \Sigma^+$, $\alpha=\sum_{i=1}^\ell n_i(\alpha)\alpha$ with $n_i(\alpha)\in \N$. Furthermore, the group $W$ is generated by the reflections $\{s_\alpha\}_{\alpha\in \Delta}$ (\cite[Ch. VI, Thm. 2]{bourbaki2007groupes}). For any $\alpha\in \Delta$, the reflection $s_\alpha$ permutes the positive roots that are not proportional to $\alpha$ (\cite[Ch. VI, Prop. 17]{bourbaki2007groupes}).\smallskip

For $\lambda\in \mathfrak{a}^*$, define \begin{equation}\label{eq:nlambda}n(\lambda)=\sum_{\begin{subarray}{c}
    \alpha\in \Sigma^+\\
    \langle \alpha,\lambda\rangle\neq 0
\end{subarray}} m(\alpha)\end{equation}and set \begin{equation}
\label{eq:kappa}\kappa(G)=\underset{\lambda \in  \mathfrak{a}^*\setminus \{0\}}{\inf} \frac{n(\lambda)}{2}.\end{equation}
We can express $\kappa$ in a more computable way. First, notice that $n$ is invariant under the Weyl group. It suffices to prove this on generators $s_\gamma$, $\gamma\in \Delta$. Then \begin{align*}
    n(s_{\gamma}\lambda) & =  \sum_{\begin{subarray}{c}
    \alpha\in \Sigma^+\setminus\{\gamma,2\gamma\}\\
    \langle \alpha,s_\gamma\lambda\rangle\neq 0
\end{subarray}} m(\alpha)+(1-\delta_{\langle\gamma,s_\gamma\lambda\rangle,0})(m(\gamma)+m(2\gamma)) \\
     & =  \sum_{\begin{subarray}{c}
    \alpha\in \Sigma^+\setminus\{\gamma,2\gamma\}\\
    \langle s_\gamma\alpha,\lambda\rangle\neq 0
\end{subarray}} m(\alpha)+(1-\delta_{-\langle \gamma,\lambda\rangle,0})(m(\gamma)+m(2\gamma)) \\
& =  \sum_{\begin{subarray}{c}
    \alpha\in \Sigma^+\setminus\{\gamma,2\gamma\}\\
    \langle \alpha,\lambda\rangle\neq 0
\end{subarray}} m(s_\gamma\alpha)+(1-\delta_{\langle \gamma,\lambda\rangle,0})(m(\gamma)+m(2\gamma)) \\
& =  \sum_{\begin{subarray}{c}
    \alpha\in \Sigma^+\setminus\{\gamma,2\gamma\}\\
    \langle \alpha,\lambda\rangle\neq 0
\end{subarray}} m(\alpha)+(1-\delta_{\langle \gamma,\lambda\rangle,0})(m(\gamma)+m(2\gamma)) \\
& =  n(\lambda)
\end{align*}using the fact that $s_\gamma$ is a permuation of $\Sigma^+\setminus \{\gamma,2\gamma\}$ and the invariance of multiplicities under the Weyl group. Notice that $m(2\gamma)$ can be zero if $2\gamma\not\in \Sigma$. Since every orbit under the Weyl group meets $\overline{\mathfrak{a}^+}$, $\kappa(G)$ is also the infimum over $\lambda\neq 0$ such that $\langle \alpha,\lambda\rangle \geq 0$ for any $\alpha\in \Sigma^+$. For such a $\lambda$, if $\alpha=\beta+\gamma$ with $\alpha,\beta,\gamma\in \Sigma^+$, then $\langle \alpha,\lambda\rangle=0$ implies $\langle \beta,\lambda\rangle=0=\langle \gamma,\lambda\rangle$. Thus, if $n(\lambda)$ is minimal, there is a unique $\alpha_i\in \Delta$ such that $\langle \alpha_i,\lambda\rangle\neq 0$. Therefore, we get \begin{equation}
    \label{kappamin} \kappa(G)=\frac{1}{2}\underset{1\leq i \leq \ell}{\min}\underset{\begin{subarray}{c}
    \alpha\in \Sigma^+\\n_i(\alpha)\geq 1
\end{subarray}}{\sum} m(\alpha).
\end{equation}The values of $\kappa$ were computed for simple groups in \cite[Section 4.3, Tab. 3]{dumas2023regularity}, in which $\kappa(G)$ is denoted $r(M)$, where $M$ is the compact dual of the symmetric space $G/K$. For the sake of completeness, we reproduce this table in Appendix \ref{sec:calculkappa} with the non-compact group $G$ as a label.\smallskip

Let $\mathfrak{n}=\bigoplus_{\alpha\in \Sigma^+} \mathfrak{g}^\alpha$. Denote $A=\exp \mathfrak{a}$, $A^+=\exp \mathfrak{a}^+$, $\overline{A^+}=\exp \overline{\mathfrak{a}^+}$ and $N=\exp \mathfrak{n}$. The Iwasawa decomposition says that \begin{equation}\label{eq:iwasawa}\fon{K\times A\times N}{G}{(k,a,n)}{kan}\end{equation}is a diffeomorphism (\cite[Thm. 6.46]{knapp2002lie}). Let $H:G\mapsto \mathfrak{a}$ be the Iwasawa projection, that is to say the unique map such that $\forall g\in G$, $g\in K\exp H(g)N$. Then $H$ is smooth.

The pair $(G,K)$ is a Gelfand pair. Let $\rho=\frac{1}{2}\sum_{\alpha\in \Sigma^+}m(\alpha)\alpha$. Let $\mathfrak{a}_\C^*$ be the space of $\R$-linear forms on $\mathfrak{a}_\C$. Then the spherical functions of $(G,K)$ are \begin{equation}\label{eq:fonctionspherique}\varphi_\lambda:g\mapsto \int_K e^{(i\lambda-\rho)(H(gk))}dk\end{equation} for $\lambda\in \mathfrak{a}_\C^*$, where $dk$ is the Haar measure on $K$ with total mass $1$ (\cite[Ch. IV, Thm. 4.3]{helgason2000groups}).

Since we want to study unitary matrix coefficients, we only want to consider positive definite spherical functions. In particular, we want to know where $\lambda$ is located when $\varphi_\lambda$ is positive-definite. A complete answer is only know in a few special cases, see for example \cite{kostant, flensted-jensen} in rank $1$. In general, it is still an open problem to completely classify such $\lambda$. However, some partial results are known. The following proposition is found in \cite[Section IV, Thm. 8.1 and B.9.(i)] {helgason2000groups} and will be sufficient for our purposes.
\begin{prop}\label{prop:defpos}Let $\lambda\in \mathfrak{a}_\C^*$, then:
\begin{itemize}
    \item $\varphi_\lambda$ is bounded if and only if $\Ima \lambda\in C=\operatorname{Conv}(W\rho)$,
    \item If $\varphi_\lambda$ positive-definite, then $\varphi_\lambda$ bounded by $\varphi(e)=1$,
    \item If $\Ima \lambda=0$, then $\varphi_\lambda$ is positive-definite.
\end{itemize}
\end{prop}

\section{\texorpdfstring{$K$}{K}-finite matrix coefficients}\label{sec:kfinite}
\subsection{\texorpdfstring{$K$}{K}-bi-invariant functions on Lie groups vs. Lie algebras}
Given the previous notations, a semisimple Lie group with finite center has a $KAK$ decomposition. This means that the $K$-bi-invariant functions can be studied as functions on $A$. The following proposition is \cite[Thm. 7.39]{knapp2002lie}.
\begin{prop}[KAK decomposition]\label{prop:KAKv1} For any $g\in G$, there are $k_1,k_2\in K$ and a unique $a\in \overline{A^+}$ such that $g=k_1ak_2^{-1}$. Furthermore, if $a\in A^+$, $k_1$ is unique up to multiplication on the right by an element of $M=Z_K(A)$.
\end{prop}
Denote $G_r=KA^+K$. Then $G_r$ is a dense open subset of $G$ which we call the set of regular elements. Let also $P:G\to \overline{\mathfrak{a}^+}$ be the map such that for any $g\in G$, $g\in K\exp P(g)K$, which is well-defined by the above proposition. Note that by definition, $P(G_r)=\mathfrak{a}^+$.

Let $\varphi$ be a $K$-bi-invariant function on $G$. Let also $\psi=\varphi\circ \exp\vert_{\overline{\mathfrak{a}^+}}$. Then $\varphi=\psi\circ P$ by $K$-bi-invariance. Hence, we can study $\psi$ and recover information on $\varphi$ using $P$. Since we want to study regularity of functions, we will use Lemma \ref{lem:precomposition}. However, $P$ is not smooth on $G$, but we will show that it is smooth on $G_r$. We will then only get results on $G_r$.

\begin{lem}\label{lem:kaksubmersion}The map $$\fonction{q}{K\times K\times \mathfrak{a}^+}{G_r}{(k_1,k_2,H)}{k_1\exp(H)k_2^{-1}}$$is a submersion.
\end{lem}
\begin{proof}If $g\in G$, denote $L_g$ and $R_g$ the translations by $g$ on the left and right respectively. Let $m:G\times G\to G$ be the multiplication map, its differential at $(a,b)$ is $$\fonction{T_{(a,b)}m}{T_aG\times T_bG}{T_{ab}G}{(X_a,X_b)}{T_aR_b(X_a)+T_bL_a(X_b)}.$$We can identify $T_gG$ with $\mathfrak{g}$ by the isomorphism $T_eL_g$. Under this identification, we have $\forall g,h\in G$, $T_hL_g=\Id$ and $T_hR_g=\Ad(g^{-1})$, so that the tangent map becomes $T_{(a,b)}m(X_a,X_b)=\Ad(b^{-1})(X_a)+X_b$. Furthermore, if $k\in K$, since $L_k(K)=K$, $T_kK\subset T_kG$ is identified with $\mathfrak{k}\subset \mathfrak{g}$. Thus by the chain rule we have $$\fonction{T_{(k_1,k_2,H)}q}{\mathfrak{k}\times \mathfrak{k}\times \mathfrak{a}}{\mathfrak{g}}{(X_1,X_2,Y)}{\Ad(k_2)(\Ad(\exp(-H))(X_1)+T_H\exp(Y))-X_2}.$$We know that $\Ad(k)$ is an isomorphism of $\mathfrak{g}$ and an isomorphism of $\mathfrak{k}$ in restriction. Furthermore, $T_H\exp:\mathfrak{a}\mapsto \mathfrak{a}$ is also an isomorphism. Thus, the map $T_{(k_1,k_2,H)}q$ is surjective if and only if $u=\Ad(k_2^{-1})\circ T_{(k_1,k_2,H)}q\circ (\Id,\Ad(k_2),(T_H\exp)^{-1})$ is surjective. We have $$u(X_1,X_2,Y)=\Ad(\exp(-H))(X_1)-X_2+Y.$$
For $\alpha\in \Sigma^+$, let $\mathfrak{k}^\alpha=\mathfrak{k}\cap (\mathfrak{g}^\alpha \oplus \mathfrak{g}^{-\alpha})$ and $\mathfrak{p}^\alpha=\mathfrak{p}\cap (\mathfrak{g}^\alpha \oplus \mathfrak{g}^{-\alpha})$. From \cite[Ch. VI, Prop. 1.4]{loos1969symmetric2}, we get $$\mathfrak{k}=\mathfrak{m}\oplus \bigoplus_{\alpha\in \Sigma^+} \mathfrak{k}^\alpha=\mathfrak{m}\oplus \mathfrak{l},$$ $$\mathfrak{p}=\mathfrak{a}\oplus \bigoplus_{\alpha\in \Sigma^+} \mathfrak{p}^\alpha=\mathfrak{a}\oplus \mathfrak{b}.$$We also get that for $\alpha\in \Sigma^+$, there exists $Z_{\alpha,1},\cdots,Z_{\alpha,m(\alpha)}$ basis of $\mathfrak{g^\alpha}$, such that setting $Z_{\alpha,i}^{+}=Z_{\alpha,i}+\theta(Z_{\alpha,i})$ and $Z_{\alpha,i}^-=Z_{\alpha,i}-\theta(Z_{\alpha,i})$, $\{Z_{\alpha,i}^{+}\}$ is a basis of $\mathfrak{k}_\alpha$ and $\{Z_{\alpha,i}^{-}\}$ is a basis of $\mathfrak{p}_\alpha$.\\
Let also $H_1,\cdots,H_\ell$ be a basis of $\mathfrak{a}$ and $Y_1,\cdots,Y_r$ a basis of $\mathfrak{m}$. Then for $H\in \mathfrak{a}$, we have $[H,Y_i]=0$, $[H,Z_{\alpha,i}^{+}]=\alpha(H)Z_{\alpha,i}^{-}$ and $[H,Z_{\alpha,i}^{-}]=\alpha(H)Z_{\alpha,i}^{+}$. Thus, we see that \begin{itemize}
    \item $u(0,0,H_i)=H_i$,
    \item $u(Y_i,0,0)=e^{-\ad(H)}(Y_i)=Y_i$,
    \item $u(0,Y_i,0)=-Y_i$,
    \item $u(Z_{\lambda,i}^{+},0,0)=e^{-\ad(H)}(Z_{\alpha,i}^{+})=\cosh(\alpha(H))Z_{\alpha,i}^{+}-\sinh(\alpha(H))Z_{\alpha,i}^{-}$,
    \item $u(0,Z_{\alpha,i}^{+},0)=-Z_{\alpha,i}^{+}$.
\end{itemize} 
Since $H\in \mathfrak{a}^+$, $\alpha(H)\neq 0$ for any $\alpha\in \Sigma^+$ and $u$ is indeed surjective.
\end{proof}

\begin{prop}\label{prop:KAKv2}The map $P:G\to \overline{\mathfrak{a}^+}$ is smooth on $G_r$. Furthermore, for each $g\in G_1$, there exists a neighbourhood $U_g$ of $g$ in $G_r$ and a choice of $g\mapsto k_i(g)$ such that $k_i$ is smooth on $U_g$, $i=1,2$ and for any $g\in U_g$, $g=k_1(g)\exp(P(g))k_2(g)^{-1}$.
\end{prop}
\begin{proof}Let $\Delta(M)=\{(m,m)\vert m\in M\}$ denote the diagonal subgroup of $K\times K$. By Lemma \ref{prop:KAKv1}, the map $$\fonction{\Tilde{q}}{(K\times K)/\Delta(M)\times \mathfrak{a}^+}{G_r}{\left((k_1,k_2) \operatorname{ mod }M,H\right)}{k_1\exp(H)k_2^{-1}}$$is a well-defined smooth bijection between manifolds of the same dimension.\\
Let $p:K\times K\to (K\times K)/\Delta(M)$ be the projection. It is a surjective submersion. Let $q$ be the submersion defined in Lemma \ref{lem:kaksubmersion}, we have $q=\Tilde{q}\circ (p,\Id)$. Thus, for any $(x,H)\in (K\times K)/\Delta(M)\times \mathfrak{a}^+$, we have $T_{(x,H)}\Tilde{q}$ surjective. But it is a linear map between vector spaces of the same dimension, so it is invertible. Thus, by the local inversion theorem and since $\Tilde{q}$ is bijective, $\Tilde{q}$ is a smooth diffeomorphism.

Let $(x,P):G_1\to (K\times K)/\Delta(M)\times \mathfrak{a}^+$ be a smooth inverse. We get that $P$ is a smooth map. From \cite[Proposition 4.26]{lee2003introduction}, since $p$ is a submersion, any $(k_1,k_2)\in K\times K$ is in the image of a smooth local section of $p$. Let $g\in G_1$, since $p$ is surjective, $x(g)=p(k_1,k_2)$. There exists a neighbourhood $V$ of $x(g)$ and a smooth section $s=(s_1,s_2):V\mapsto K\times K$ such that $s(x(g))=(k_1,k_2)$.\\
Let $U=x^{-1}(V)$ be a neighbourhood of $g$, then $k_i=s_i\circ x$ is smooth on $U$ and $g=k_1(g)\exp(P(g))k_2(g)^{-1}$.
\end{proof}

\begin{coro}\label{coro:liealgvsliegroup}Let $\varphi$ be a $K$-bi-invariant function on $G$, then $\varphi\in C^{(r,\delta)}(G_r)$ if and only if $\varphi\circ \exp\in C^{(r,\delta)}(\mathfrak{a}^+)$.\\Furthermore, if $\left(\varphi_i\right)_{i\in I}$ is a family of $K$-bi-invariant function on $G$, then $\left(\varphi_i\right)$ is bounded in $C^{(r,\delta)}(G_r)$ if and only if $\left(\varphi_i\circ \exp\right)$ is bounded in $C^{(r,\delta)}(\mathfrak{a}^+)$.
\end{coro}
\begin{proof}Since $\exp$ is smooth, the first implication is a consequence of Lemma \ref{lem:precomposition}. For the converse, assume $\psi=\varphi\circ \exp\vert_{\mathfrak{a}^+} \in C^{(r,\delta)}(\mathfrak{a}^+)$. By the previous proposition, the map $P$ is smooth on $G_r$ and $\varphi=\psi\circ P$ by $K$-bi-invariance, thus $\varphi\in C^{(r,\delta)}(G_r)$ by Lemma \ref{lem:precomposition}.
\end{proof}

\subsection{\texorpdfstring{$K$}{K}-finite matrix coefficients}\label{sec:kfin}
In this section, we want to show that the optimal regularity of all $K$-bi-invariant matrix coefficients of unitary representations of $G$ coincides with the optimal regularity of all $K$-finite matrix coefficients of unitary representations. The ideas are very similar to \cite[Section 5]{dumas2023regularity} where the same result is shown for compact symmetric pairs. However, this proof relied heavily on \cite[Lemma 2.2]{dLMdlS}, which only works in the compact case. Thus the first step of our proof is to show a similar lemma for $G$ non-compact.

\begin{definition}Let $\pi$ be a unitary representation of $G$ on $\mathcal{H}$ and $(\rho,V)$ a representation of $K$. We say that $\xi\in\mathcal{H}$ is \begin{itemize}
    \item $K$-finite if $\vspan(\pi(K)\xi)$ is finite dimensional,
    \item of $K$-type $V$ if $\vspan(\pi(K)\xi)\simeq V$ as a representation of $K$.
\end{itemize}
\end{definition}
Note that this definition of $K$-type $V$ is not standard.

\begin{lem}\label{lem:densityirredcoef}
Let $G$ be a second countable locally compact group, $L$ a compact subset of $G$. Then the vector space generated by matrix coefficients of irreducible unitary representations of $G$ restricted to $L$ is dense in $C(L,\mathbb{C})$.
\end{lem}
\begin{proof}Let $S$ be the set of matrix coefficients of irreducible unitary representations. Let $E$ be the vector space generated by $S$. By Gelfand-Raikov's theorem, the $*$-algebra $A$ generated by $S$ is dense in $C(L)$.

Assume that $E$ is not dense. Then, by Hahn-Banach's theorem, there exists a linear form $f\in C(L)^*$ such that $f\neq 0$ and $f\vert_E=0$. By density of $A$, there exists a matrix coefficient $\varphi$ such that $f(\varphi)\neq 0$. Since $G$ is second countable locally compact, there are $X$ a standard Borel space and $\mu$ a $\sigma$-finite measure on $X$ such that $\forall g\in G$, $$\varphi(g)=\int_X \varphi_x(g)d\mu(x)$$with $\varphi_x$ a coefficient of an irreducible unitary representation, hence an element of $S$ (\cite[Section 8.4]{kirillov1976elements}). Furthermore, if $\varphi_x(g)=\langle \pi(g)\xi_x,\eta_x\rangle$, then $x\mapsto \Vert \xi_x\Vert \Vert \eta_x\Vert \in L^1(X,\mu)$.

There exists a complex finite measure $\nu$ on $L$ such that $f(\varphi)=\int_L \varphi(g)d\nu(g)$. Then $$\int_L\int_X \vert\varphi_x(g)\vert d\mu(x)d\vert\nu\vert(g)\leq \vert\nu\vert(L) \int_X \Vert \xi_x\Vert \Vert \eta_x\Vert d\mu(x) < +\infty.$$
Thus by Fubini's theorem, we have \begin{align*}
   f(\varphi)  & = \int_L\int_X \varphi_x(g)d\mu(x)d\nu(g) \\
     & =\int_X\int_L \varphi_x(g) d\nu(g)d\mu(x)\\
     & =  \int_X f(\varphi_x)d\mu(x)\\
     & =  \int_X 0 \,d\mu(x)
\end{align*}
which is a contradiction.
\end{proof}

Let $G$ be a semisimple Lie group with finite center and $K$ a maximal compact subgroup. Let $U=K\times K$ and $(\rho,V)$ an irreducible unitary representation of $U$ (hence finite-dimensional). The group $U$ acts on $G$ by conjugation. For $g\in G$, let $U_g=\Stab(g)$, $V_g=V^{U_g}$ and $P_g$ the orthogonal projection on $V_g$.
\begin{lem}\label{lem:smoothkfintokinv}
For any $g_0\in G$, there exists a smooth function $\psi:G\to B(V)$
 such that \begin{enumerate}
     \item $\forall u\in U,g\in G$, $\psi(u.g)=\psi(g)\circ \rho(u)^{-1}$,
     \item $\forall v_1,v_2\in V$, $g\mapsto \langle \psi(g)v_1,v_2\rangle$ is a matrix coefficient of a unitary representation of $G$,
     \item $\psi(g_0)=P_{g_0}$.
 \end{enumerate}
\end{lem}

\begin{proof}Let $F$ be the set of functions $\phi:G\mapsto B(V)$ such that for any $v_1,v_2\in V$, $g\mapsto \langle \phi(g)v_1,v_2\rangle$ is a $K$-finite coefficient of a finite direct sum of irreducible unitary representations of $G$. Such a coefficient is a finite sum of $K$-finite matrix coefficients of irreducible unitary representations, so it is smooth (\cite[Thm. 8.1, Prop. 8.5]{knapp2001representation}). Hence any $\phi\in F$ is smooth, and verifies $(2)$.

If $\phi\in F$, define $\psi(g)=\int_U \phi(u.g)\rho(u)du$. Then clearly $\psi$ verifies $(1)$. Let $e_1,\cdots,e_d$ be an orthonormal basis of $V$. Let $\pi_{ij}$ be a unitary representation and $\xi_{ij},\eta_{ij}$ be $K$-finite vectors such that $$\forall g\in G, \langle \phi(g)e_i,e_j\rangle=\langle \pi_{ij}(g)\xi_{ij},\eta_{ij}\rangle.$$

Let $u=(k_1,k_2)\in U$, then $\rho(u)=\rho(k_1,1_K)\rho(1_K,k_2)$. We define functions $K\to \mathbb{C}$ such that $$\rho(1_K,k)e_i=\sum_{j=1}^d \lambda_{ij}(k)e_j$$and$$\rho(k,1_K)e_i=\sum_{j=1}^d \mu_{ij}(k)e_j.$$Then \begin{multline*}
   \langle \psi(g)e_i,e_j\rangle   \\ \begin{aligned}&= \int_U \langle \phi(u.g)\rho(u)e_i,e_j\rangle du \\
     & =  \int_{K\times K}\sum_{p,q=1}^d \mu_{pq}(k_1)\lambda_{ip}(k_2) \langle \phi(k_1gk_2^{-1})e_q,e_j\rangle dk_1dk_2\\
     & =  \int_{K\times K}\sum_{p,q=1}^d \mu_{pq}(k_1)\lambda_{ip}(k_2) \langle \pi_{qj}(k_1gk_2^{-1})\xi_{qj},\eta_{qj}\rangle dk_1dk_2\\
     & = \sum_{p,q=1}^d\int_{K\times K} \langle \pi_{qj}(g)\left(\lambda_{ip}(k_2)\pi_{qj}(k_2^{-1}) \xi_{qj}\right),\overline{\mu_{pq}(k_1)}\pi_{qj}(k_1^{-1})\eta_{qj}\rangle dk_1dk_2\\
     & =  \sum_{p,q=1}^d \left\langle \pi_{qj}(g)\left(\int_K \lambda_{ip}(k_2)\pi_{qj}(k_2^{-1})\xi_{qj}dk_2\right),\left(\int_K \overline{\mu_{pq}(k_1)}\pi_{qj}(k_1^{-1})\eta_{qj} dk_1\right) \right\rangle
     \end{aligned}
\end{multline*}
Now since $\xi_{qj}$ is $K$-finite, the vector $\Tilde{\xi}_{qj}=\int_K \lambda_{ip}(k_2)\pi_{qj}(k_2^{-1})\xi_{qj}dk_2$ is $K$-finite, and similarly, $\Tilde{\eta}_{qj}=\int_K \overline{\mu_{pq}(k_1)}\pi_{qj}(k_1^{-1})\eta_{qj} dk_1$ is $K$-finite. For each $q,j$, $\pi_{qj}$ is a finite direct sum of irreducible representation, so $g\mapsto  \langle \psi(g)e_i,e_j\rangle$ is a $K$-finite coefficient of $\oplus_{p,q=1}^d \pi_{qj}$ which is a finite direct sum of irreducible representations. By linearity, this remains true for any $v_1,v_2\in V$, thus we showed that if $\phi\in F$, so is $\psi$. Hence, $\psi$ is smooth and verifies $(2)$.\smallskip

It remains to show that there exists $\phi\in F$ such that $\psi(g_0)=P_{g_0}$. Notice that if $u\in U_{g_0}$, then $$\psi(g_0)=\psi(u.g_0)=\psi(g_0)\circ \rho(u)^{-1}.$$
Thus $V_{g_0}^\bot=\sum_{u\in U_{g_0}}\Ima(\rho(u)-\Id)\subset \ker \psi(g_0)$.

First, let us find $\phi\in F$ such that $\rank\psi(g_0)=\dim V_{g_0}$. Consider $O\simeq U/U_{g_0}$ the $U$-orbit of $g_0$ in $G$. Let $s$ be a measurable section, that is to say $s:O\mapsto U$ such that $s(u.g_0)\in uU_{g_0}$. Let $\phi:O\mapsto B(V)$ be the map $x\mapsto \rho(s(x))^{-1}$. Then $\psi:x\mapsto \int_U \rho(s(u.x)^{-1}u)du$ is such that $\psi(g_0)$ is the identity on $V_{g_0}$, and by the above discussion $0$ on $V_{g_0}^\bot$. Thus $\psi(g_0)=P_{g_0}$. Let $\mu$ be the image of the Haar measure on $O$ by the map $p:u\mapsto u.g_0$. Then $\phi\in L^1(X;B(V),\mu)$. By density of continuous function, there are continuous maps $f:O\to B(V)$ arbitrarily close to $\phi$ in $\Vert.\Vert_1$. But then, \begin{align*}
    \left\Vert \int_U f(u.g_0)\rho(u)du-\int_U \phi(u.g_0)\rho(u)du \right\Vert & \leq  \int_U \Vert f(u.g_0)-\phi(u.g_0)\Vert du  \\
     & \leq  \int_U \Vert (f-\phi)\circ p\Vert du\\
     & \leq  \int_O \Vert f-\phi\Vert d\mu\\
     & \leq  \Vert f-\phi\Vert_1.
\end{align*}

So we can take $f$ close enough so that $\rank \int_U f(u.g_0)\rho(u)du=\rank P_{g_0}$. Then since $O$ is closed in $G$ normal, by Tietze extension theorem, we can extend $f$ to a continuous map $\phi:G\mapsto B(V)$.

Let $L$ be a compact subset of $G$ containing $g_0$ and $\varepsilon>0$. Let $\phi_{ij}:g\mapsto \langle \phi(g)e_i,e_j\rangle$. Then by Lemma \ref{lem:densityirredcoef}, there exists $\pi_{ij}$ a finite direct sum of irreducible representations of $G$ and $\varphi_{ij}:g\mapsto \langle \pi_{ij}\xi_{ij},\eta_{ij}\rangle$ such that \begin{equation}\label{eq:approxirred}
    \underset{g\in L}{\sup} \vert \phi_{ij}(g)-\varphi_{ij}(g)\vert \leq  \varepsilon.
\end{equation}\\
But by the Peter-Weyl theorem, $K$-finite vectors are dense in the representation space of $\pi_{ij}$. Thus, there are $\Tilde{\xi}_{ij}$ and $\Tilde{\eta}_{ij}$ $K$-finite such that $$\Vert \xi_{ij}-\Tilde{\xi}_{ij}\Vert \leq \min\left(\Vert\xi_{ij}\Vert,\frac{\varepsilon}{\Vert \eta_{ij}\Vert}\right)$$and$$\Vert \eta_{ij}-\Tilde{\eta}_{ij}\Vert \leq \frac{\varepsilon}{\Vert \xi_{ij}\Vert}.$$
Thus, setting $\Tilde{\varphi}_{ij}(g)=\langle \pi_{ij}(g)\Tilde{\xi}_{ij},\Tilde{\eta}_{ij}\rangle$, we have for $g\in L$, \begin{equation}\label{eq:approxkfini}\begin{aligned}
   \vert \varphi_{ij}(g)-\Tilde{\varphi}_{ij}(g) \vert & \leq  \vert \langle \pi_{ij}{\xi}_{ij},\eta_{ij}-\Tilde{\eta}_{ij}\rangle\vert + \vert \langle \pi_{ij}(\xi_{ij}-\Tilde{\xi}_{ij}),\Tilde{\eta}_{ij}\rangle\vert  \\
     & \leq  \Vert \xi_{ij}\Vert \Vert \eta_{ij}-\Tilde{\eta}_{ij}\Vert + \Vert \Tilde{\eta}_{ij}\Vert \Vert \xi_{ij}-\Tilde{\xi}_{ij}\Vert\\
     & \leq  \varepsilon+ \frac{\varepsilon}{\Vert \xi_{ij}\Vert} \Vert \xi_{ij}-\Tilde{\xi}_{ij}\Vert + \Vert \eta_{ij}\Vert \Vert \xi_{ij}-\Tilde{\xi}_{ij}\Vert\\
     & \leq  3\varepsilon
\end{aligned}\end{equation}
Thus with \eqref{eq:approxirred} and \eqref{eq:approxkfini}, we get \begin{equation}\label{eq:approxirr+kfini}\underset{g\in L}{\sup} \vert \phi_{ij}(g)-\Tilde{\varphi}_{ij}(g)\vert \leq  4\varepsilon.\end{equation}

Now for $v_1=\sum_{i=1}^d x_ie_i$ and $v_2=\sum_{i=1}^d y_ie_i$, we have $$\langle \phi(g)v_1,v_2\rangle=\sum_{i,j}x_i\overline{y_j}\phi_{ij}(g).$$Let $\varphi(g)$ be defined as the linear map on $V$ such that $\varphi(g)e_i=\sum_{j=1}^d \Tilde{\varphi}_{ij}(g)e_j$. Let $\pi=\bigoplus_{i,j} \pi_{ij}$, then $\xi=(x_i\xi_{ij})$ and $\eta=(y_j\eta_{ij})$ are two $K$-finite vectors and $\varphi:G\to B(V)$ is such that $$\langle \varphi(g)v_1,v_2\rangle=\sum_{i,j} x_i\overline{y_j}\Tilde{\varphi}_{ij}=\langle \pi(g)\xi,\eta\rangle.$$So we have $\varphi\in F$. Furthermore, for $g\in L$, \begin{align*}
    \Vert \phi(g)-\varphi(g)\Vert & = & \underset{\Vert v_1\Vert_2=\Vert v_2\Vert_2=1}{\sup} \vert \langle (\phi(g)-\varphi(g))v_1,v_2\rangle   \\
     & \leq & \underset{i,j}{\max } \vert \phi_{ij}(g)-\varphi_{ij}(g)\vert \underset{\Vert v_1\Vert_2=\Vert v_2\Vert_2=1}{\sup} \Vert v_1 \Vert_1\Vert v_2\Vert_1\\
     & \leq & 4C\varepsilon
\end{align*}where $C$ depends only on $V$. Thus, this shows than we can find $\varphi$ in $F$ arbitrarily close to $\phi$ on any compact subset containing $g_0$, in particular on the orbit $O$. Thus, $\int_U \varphi(u.g_0)\rho(u)du$ is of rank $\dim V_{g_0}$ for $\varepsilon$ small enough.

Finally, we get $\phi\in F$ such that $\psi(g_0)$ has rank $\dim V_{g_0}$ and is zero on $V_{g_0}^\bot$. Thus there is $A\in B(V)$ such that $A\psi(g_0)=P_{g_0}$. Replace $\phi$ by $A\phi$ and we get the result.
\end{proof}

Let $\pi$ be a unitary representation of $G$ on $\mathcal{H}$ and $\xi,\eta\in\mathcal{H}$ of $K$-type $V,W$ respectively, for $V,W$ irreducible representations of $K$. Denote $V_\xi=\vspan(\pi(K)\xi)$. Then there is an isomorphism $i_\xi:V\to V_\xi\subset \mathcal{H}$, denote $\xi_0=i_\xi^{-1}(\xi)$. Similarly, define $V_\eta$ and $i_\eta$. Then the map \begin{equation}\label{eq:defi_f}\fonction{f}{B(\mathcal{H})}{L(V,W^*)\simeq V^*\otimes W^*}{A}{i_\eta^*Ai_\xi}\end{equation}is $K\times K$ equivariant.

For the associated matrix coefficient, we have $\varphi(g)=\langle\pi(g)\xi,\eta\rangle=\langle f(\pi(g))\xi_0,\eta_0\rangle$.

Now denote $(\rho,V_\rho)$ the irreducible representation of $U=K\times K$ on $V^*\otimes W^*$. The $U$-equivariance of $f$ means that for any $(k,k')\in U$ and $A\in B(\mathcal{H})$, we have \begin{equation}\label{eq:equiv_f}f(\pi(k)A\pi(k')^{-1})=\rho(k,k')(f(A)).\end{equation}Furthermore, there are $v_1,\cdots,v_n\in V_\rho$ and $\xi_1,\cdots,\xi_n,\eta_1,\cdots,\eta_n\in\mathcal{H}$ such that\begin{equation}\label{eq:dec_f}f(A)=\sum_{i=1}^n \langle A\xi_i,\eta_i\rangle v_i.\end{equation}

\begin{prop}\label{prop:regktypeV}If for any $K$-bi-invariant matrix coefficient $\varphi$ of a unitary representation of $G$, the function $\varphi\circ\exp$ is in $C^{(r,\delta)}(\mathfrak{a}^+)$, then the map $f\circ \pi$ is in $C^{(r,\delta)}(G_r)$.
\end{prop}
\begin{proof}
Let $g_0\in G_r$ and $\psi$ given by Lemma \ref{lem:smoothkfintokinv} for the representation $(\rho,V_\rho)$. Let $\Tilde{f}:g\mapsto \psi(g)(f(\pi(g)))$. By \eqref{eq:equiv_f} and $(1)$ of Lemma \ref{lem:smoothkfintokinv}, we have \begin{equation}\label{eq:kbiinv_tildef}\Tilde{f}(u.g)=\psi(u.g)(f(\pi(u.g)))=\psi(g)\rho(u)^{-1}\rho(u)(f(\pi(g))=\Tilde{f}(g)\end{equation}
so $\Tilde{f}$ is a $K$-bi-invariant map.

Let $(e_1,\cdots,e_d)$ be an orthonormal basis of $V_\rho$, by $(2)$ of Lemma \ref{lem:smoothkfintokinv} there are $(\pi_{ij},\mathcal{H}_{ij})$ unitary representations of $G$ and $a_{ij},b_{ij}\in \mathcal{H}_{ij}$ such that $$\langle\psi(g)v_i,e_j\rangle=\langle\pi_{ij}(g)a_{ij},b_{ij}\rangle$$so $\psi(g)v_i=\sum_{j=1}^d \langle\pi_{ij}(g)a_{ij},b_{ij}\rangle e_j$ and finally with \eqref{eq:dec_f}, \begin{equation}\label{eq:tildef}\Tilde{f}(g)=\sum_{i=1}^n\sum_{j=1}^d \langle(\pi_{ij}\otimes \pi)(g)(a_{ij}\otimes\xi_i),b_{ij}\otimes \eta_i\rangle e_j.\end{equation}

Hence, $\Tilde{f}$ is a sum of $K$-bi-invariant matrix coefficients of unitary representations of $G$, so by the hypothesis and Corollary \ref{coro:liealgvsliegroup}, $\Tilde{f}\in C^{(r,\delta)}(G_r)$.\smallskip

By Lemma \ref{prop:KAKv1}, if $a\in A^+=\exp \mathfrak{a}^+$, we have $U_a=\Delta(M)$. Thus, $V_a=V_0$ is independent of $a\in A^+$. If $g=(k_1,k_2).a=k_1ak_2^{-1}$, we have $(k,k')\in U_g$ if an only if $(k_1^{-1}kk_1,k_2^{-1}k'k_2)\in \Delta(M)$ and so $V_g=\rho(k_1,k_2)V_0$.

Let $g_0=k_0a_0k_0^{'-1}$ and $V_1=V_{g_0}$. Since $\psi(g_0)=P_{g_0}$, there is an orthonormal basis adapted to $V_1$ such that $$\psi(g_0)=\begin{pmatrix}\Id&0\\0&0\end{pmatrix}.$$ Furthermore, since $\psi$ is smooth, there is $A_{g_0}$ neighbourhood of $g_0$ such that $$\psi(g)=\begin{pmatrix}A(g)&*\\ *&*\end{pmatrix}$$
with $g\mapsto A(g)$ smooth, $A(g)$ invertible for any $g\in A_{g_0}$. Up to restricting $A_{g_0}$, by Proposition \ref{prop:KAKv2}, we have $g=k_1(g)\exp(P(g))k_2(g)^{-1}=k_1(g)a(g)k_2(g)^{-1}$ with $k_1,k_2$ smooth on $A_{g_0}$.

By \eqref{eq:kbiinv_tildef}, for any $g\in A_{g_0}$, we have $$\Tilde{f}(g)=\Tilde{f}(a(g))=\Tilde{f}(k_0a(g)k_0^{'-1}).$$
But then $f(\pi(k_0a(g)k_0^{'-1}))\in V_{k_0a(g)k_0^{'-1}}=\rho(k_0,k_0')V_0=V_1$. Set $$\Phi(g)=\rho(k_1(g)k_0^{-1},k_2(g)k_0^{'-1})\begin{pmatrix}
    A(k_0a(g)k_0^{'-1})^{-1} & 0\\0&0
\end{pmatrix},$$ it is a smooth map on $A_{g_0}$ because $A$ is smooth invertible, $k_1,k_2$ are smooth and $\rho$ is a finite dimensional representation of $U$ thus smooth. Since $f(\pi(k_0a(g)k_0^{'-1}))\in V_1$, we have \begin{align*}
    \Phi(g)(\Tilde{f}(g)) & =  \Phi(g)(\Tilde{f}(\pi(k_0a(g)k_0^{'-1})))  \\
     & =  \Phi(g)\psi(g)(f(\pi(k_0a(g)k_0^{'-1})))\\
     & =  \rho(k_1(g)k_0^{-1},k_2(g)k_0^{'-1})(f(\pi(k_0a(g)k_0^{'-1})))\\
     & =  f(\pi(k_1(g)a(g)k_2(g)^{-1}))\\
     & =  f(\pi(g))
\end{align*}

Now let $B:B(V)\times V\to V$ be the bilinear map sending $(u,v)$ to $u(v)$. We showed that on $A_{g_0}$, $f\circ \pi=B\circ (\Phi,\Tilde{f})$. Since $\Phi$ is smooth on $A_{g_0}$ and $\Tilde{f}\in C^{(r,\delta)}(G_r)$, we get by Leibniz formula that $f\circ g\in C^{(r,\delta)}(A_{g_0})$.

So for any $g_0\in G_r$, there exists a neighbourhood $A_{g_0}$ such that $f\circ \pi \in C^{(r,\delta)}(A_{g_0})$. Thus, $f\circ \pi \in C^{(r,\delta)}(G_r)$.
\end{proof}

\begin{theorem}\label{thm:kfinite}The optimal regularity of $K$-bi-invariant matrix coefficient of unitary representations of $G$ on $G_r$ is equal to the optimal regularity of $K$-finite matrix coefficients of unitary representations of $G$ on $G_r$.
\end{theorem}
\begin{proof}One inequality is trivial since $K$-bi-invariant coefficients are $K$-finite.

For the other inequality, let $(r,\delta)$ such that any $K$-bi-invariant matrix coefficient of unitary representations $G$ is in $C^{(r,\delta)}(G_r)$. Let $\varphi:g\mapsto \langle \pi(g)\xi,\eta \rangle$ be a $K$-finite matrix coefficient of a unitary representation.

If $\xi,\eta$ are of $K$-type $V,W$ respectively, with $V,W$ irreducible representations of $K$, we showed that $\varphi(g)=\langle\pi(g)\xi,\eta\rangle=\langle f(\pi(g))\xi_0,\eta_0\rangle$ and in Proposition \ref{prop:regktypeV} that $f\circ \pi\in C^{(r,\delta)}(G_r)$, thus $\varphi\in C^{(r,\delta)}(G_r)$.

For the general case, if $\xi,\eta$ are $K$-finite, $V_\xi,V_\eta$ are finite dimensional representations of $K$, so they decompose into a finite number of irreducible representations. Thus, $\varphi$ is a finite sum of matrix coefficient of the previous case, so $\varphi\in C^{(r,\delta)}(G_r)$.
\end{proof}

\section{Boundedness of positive definite spherical functions}\label{sec:mainsec}
In this section, we study the spherical functions of the pair $(G,K)$ and their Hölder norms. We will use the estimates obtained by Duistermaat, Kolk and Varadarajan in \cite{DKV} using the method of stationary phase. We will denote $\psi_\lambda=\varphi_\lambda\circ \exp\vert_{\mathfrak{a}}$ the spherical functions on the Lie algebra.
\begin{theorem}\label{thm:reg}Let $G$ be a connected semisimple Lie group with finite center and $K$ a maximal compact subgroup. Let $r=\lfloor \kappa(G)\rfloor$, $\delta=\kappa(G)-r$. Let $C$ be a bounded subset of $\mathfrak{a}^*$. Then the family of spherical functions $\psi_\lambda$ of $(G,K)$ with $\Ima \lambda\in C$ is bounded in $C^{(r,\delta)}(\mathfrak{a}^+)$.
\end{theorem}
\begin{remark}By Proposition \ref{prop:defpos}, this implies that the family of bounded spherical functions, thus the subfamily of positive-definite spherical functions, is bounded in $C^{(r,\delta)}(\mathfrak{a}^+)$.
\end{remark}
\begin{proof}For $\lambda\in \mathfrak{a}_\C^*$, we write $\lambda=\xi+i\eta$ with $\xi,\eta\in \mathfrak{a}^*$ and for $t\in \R$, $\lambda_t=t\xi+i\eta$.\\Denote $f(\xi,\eta,t,Y,k)=e^{(i\lambda_t-\rho)(H(\exp(Y)k))}\in C^\infty(\mathfrak{a}^*\times \mathfrak{a}^*\times \R \times \mathfrak{a} \times K)$. We denote $D$ the operator which differentiates a function with respect to the variable $Y\in \mathfrak{a}$. Let $s\in \N$. Then for any $Y\in \mathfrak{a}$, $X=(X_1,\cdots,X_s)\in \mathfrak{a}^s$, \begin{equation}\label{eq:dpsi_1}D^s\psi_{\lambda_t}(Y)(X)=\int_K D^sf(\xi,\eta,t,Y,k)(X)dk.\end{equation}By induction on $s$, there is a polynomial $P\in C^\infty(\mathfrak{a}^*\times \mathfrak{a}^*\times \mathfrak{a} \times K\times \mathfrak{a}^s)[t]$ of degree $s$ such that \begin{equation}\label{eq:dpsi_2}D^sf(\xi,\eta,t,Y,k)(X)=P(t)e^{it\xi(H(\exp(Y)k))}.\end{equation}

For $0\leq j \leq s$, let $g_j(\xi,\eta,Y,X)\in C^\infty(K)$ be defined by \begin{equation}\label{eq:dpsi_g}g_j=\left.\frac{1}{j!} \frac{d^j}{dt^j}\left( (D^sf)e^{-it\xi(H(\exp(Y)k))} \right)\right\vert_{t=0}.\end{equation}Then we have \begin{equation}\label{eq:derivphient}D^s\psi_{\lambda_t}(Y)(X)=\sum_{j=0}^s t^j \int_K e^{it\xi(H(\exp(Y)k))}g_j(\xi,\eta,Y,X)(k)dk.\end{equation}

Let $S$ be the unit sphere in $\mathfrak{a}^*$ and $L$ a compact subset of $\mathfrak{a}^+$. Up to replacing $L$ by its convex hull, which is still a compact subset of $\mathfrak{a}^+$ by Carathéodory's theorem and the convexity of $\mathfrak{a}^+$, we can assume that $L$ is convex. We consider $C^\infty(K)$ endowed with the topology given by the family of seminorms $p_i:g\mapsto \underset{k\in K}{\sup} \Vert D^ig(k)\Vert$.\\
For $(Y,\xi)\in L\times S$, by \cite[Prop. 9.2]{DKV}, there exists a neighbourhood $V_{Y,\xi}$ of $Y,\xi$ in $\mathfrak{a}\times \mathfrak{a}^*$ and a continuous seminorm $v_{Y,\xi}$ on $C^\infty(K)$ such that for any $(Y',\xi')\in V_{Y,\xi}$, $g\in C^\infty(K)$, $t\geq 1$, \begin{equation}\label{eq:statphase}\left\vert \int_K e^{it\xi'(H(\exp(Y')k))}g(k)dk\right\vert \leq v_{Y,\xi}(g) t^{-n(\xi)/2}\leq v_{Y,\xi}(g) t^{-\kappa(G)}.\end{equation}

Now write $\mathfrak{a}^*\times C= ([0,1]S\times C)\cup (\R_{\geq 1}S\times C)$.\\
For any $(\xi,\eta)\in S\times C$, $0\leq t\leq 1$, $Y\in L$, then \eqref{eq:derivphient} implies \begin{equation}
    \label{eq:closeto0} \begin{aligned}
        \Vert D^s\psi_{\lambda_t}(Y)\Vert  & = \underset{\Vert X_i\Vert=1}{\sup} \vert D^s\psi_{\lambda_t}(Y)(X)\vert  \\
         & \leq  \underset{\Vert X_i\Vert=1}{\sup}\sum_{j=0}^s t^j \int_K \Vert g_j(\xi,\eta,Y,X)\Vert_\infty dk\\
         & \leq  \underset{\begin{subarray}{c}
  \Vert X_i\Vert=1\\
  Y\in L,\xi\in S,\eta\in C,k\in K
  \end{subarray}}{\sup} \sum_{j=0}^s \vert g_j(\xi,\eta,Y,X)(k) \vert\\
  & \leq  C_{L,s,C}
    \end{aligned} 
\end{equation}
where $C_{L,s,C}>0$ is independent from $Y\in L$, $\xi\in S$, $\eta\in C$, using that $g_j$ is smooth hence bounded on compact subsets.

On the other hand, write $L\times S=\underset{(Y,\xi)\in L\times S}{\bigcup} V_{Y,\xi}$. Since $L\times S$ is compact, there exists a finite subcover $U_{Y_1,\xi_1},\cdots,U_{Y_n,\xi_n}$. Then for any $(\xi,\eta)\in S\times C$, $t\geq 1$, $Y\in L$, \begin{equation}
    \label{eq:awayfrom0} \begin{aligned}
        \Vert D^s\psi_{\lambda_t}(Y)\Vert  & = \underset{\Vert X_i\Vert=1}{\sup} \vert D^s\varphi_{\lambda_t}(Y)(X)\vert  \\
         & \leq  \underset{\Vert X_i\Vert=1}{\sup}\sum_{j=0}^s t^j \left\vert\int_K e^{it\xi(H(\exp(Y)k))} g_j(\xi,\eta,Y,X)(k)dk\right\vert \\
         & \leq \sum_{j=0}^s t^j\underset{1\leq i\leq n}{\max}\underset{\begin{subarray}{c}
  \Vert X_i\Vert=1\\
  Y\in L,\xi\in S,\eta\in L
  \end{subarray}}{\sup} v_{Y_i,\xi_i}(g_j(\xi,\eta,Y,X))t^{-\kappa(G)}\\
  & \leq  D_{L,s,C}t^{s-\kappa(G)} 
    \end{aligned}
\end{equation}

where $D_{L,s,C}>0$ is independent from $Y\in L$, $\xi\in S$, $\eta\in C$. Again, the last line comes from the fact that $g_j$ is smooth hence any of its differentials is bounded on compact subsets, and the topology on $C^\infty(K)$ is given by the seminorms $p_i$, thus if $v$ is a continuous seminorms, there is $i$ such that $v(g)\leq Cp_i(g)=C\underset{k\in K}{\sup} \Vert D^ig(k)\Vert$.

Thus combining \eqref{eq:closeto0} and \eqref{eq:awayfrom0}, for any $\lambda$ such that $\Ima \lambda\in C$, $H\in L$, $s\leq r$, \begin{equation}\label{eq:finalineq_integer}\Vert D^s\psi_{\lambda}(Y)\Vert \leq \max(C_{L,s,C},D_{L,s,C})=M_{L,s}.\end{equation}

Thus the differentials of the family of spherical functions are bounded independently on $\lambda$ such that $\Ima \lambda\in C$ up to order $r$. If $r=\kappa(G)$, the proof is complete.

Otherwise, $\kappa(G)-r=\frac{1}{2}$. Then using \eqref{eq:awayfrom0} for $s=r$ and $s=r+1$, we show that for any $x,y\in L$, $\xi\in S$,$\eta\in C$, $t\geq 1$, we have on the one hand \begin{equation}\label{eq:holderstep1}\Vert D^r \psi_{\lambda_t}(x)-D^r\psi_{\lambda_t}(y)\Vert \leq \Vert D^r \psi_{\lambda_t}(x)\Vert+\Vert D^r\psi_{\lambda_t}(y)\Vert\leq 2D_{L,r,C}t^{-1/2}\end{equation}and on the other hand, since $L$ is convex, we get by the mean value theorem that \begin{equation}\label{eq:holderstep2}
    \Vert D^r \psi_{\lambda_t}(x)-D^r\psi_{\lambda_t}(y)\Vert \leq \left(\underset{Y\in L}{\sup}\Vert {D^{r+1}}\psi_{\lambda_t}(Y)\Vert\right)\Vert x-y\Vert \leq D_{L,r+1,C}t^{1/2}\Vert x-y\Vert.\end{equation}
Thus, combining \eqref{eq:holderstep1} and \eqref{eq:holderstep2} yields \begin{equation}\label{eq:holderstep3}\Vert D^r \psi_{\lambda_t}(x)-D^r\psi_{\lambda-t}(y)\Vert \leq \left(2D_{L,r,C}D_{L,r+1,C}\right)^{1/2}\Vert x-y\Vert^{1/2}.\end{equation}
Hence, setting $M_L=\max\left(\left(2D_{L,r,C}D_{L,r+1,C}\right)^{1/2},C_{L,r+1,C}(\diam L)^{1/2}\right)$, we have that for any $\lambda$ such that $\Ima \lambda \in C$, $x,y\in L$, \begin{equation}\label{eq:holderstepfin}\Vert D^r \psi_{\lambda}(x)-D^r\psi_\lambda(y)\Vert \leq M_L \Vert x-y\Vert^{\kappa(G)-r}.\qedhere\end{equation}
\end{proof}

\begin{remark}If $\Ima \lambda$ is allowed to be unbounded, then the functions are not even bounded in $C(\mathfrak{a}^+)$.
\end{remark}

We will now show that this result is optimal. The bounds used in the previous theorem are not sharp in general, but for a subfamily where $n(\xi)=\frac{\kappa(G)}{2}$, they are. Thus, we will show that this particular subfamily is already unbounded in higher regular Hölder spaces.
\begin{theorem}\label{thm:opti}We keep the notations of Theorem \ref{thm:reg}. For any $\delta'>\delta$, the family of positive definite spherical functions of $(G,K)$ is not bounded in $C^{(r,\delta')}(\mathfrak{a}^+)$.
\end{theorem}

Before this, we first prove a lemma showing that complex exponentials are not bounded in Hölder spaces. We will reduce the problem for spherical functions of $(G,K)$ to such functions.
\begin{lem}\label{lem:expo}Let $E$ be a finite dimensional real vector space, $U$ an open subset of $E$ such that $0\in \overline{U}$. Let $u_1,\cdots,u_n\in E^*$ distinct and non-zero, and $f_1,\cdots,f_n:E\to \C$ continuous functions such that for any $U'$ open subset of $U$, there is $x\in U'$ such that $\sum \vert f_j(x)\vert\neq 0$. Then there exists $C>0$, $d>0$, $x\in U$, and an open set $V$ with $0\in \overline{V}$ such that for all $y=x+h,h\in V$, $m\in \N$ and $N\geq \frac{d}{\Vert h\Vert}$, $$\frac{1}{N}\sum_{t=m}^{m+N-1} \left\vert \sum_{j=1}^n  f_j(x)e^{itu_j(x)}-f_j(y)e^{itu_{j}(y)} \right\vert^2 \geq C.$$
\end{lem}
\begin{proof}
Up to multiplying $f_j$ by $x\mapsto e^{imu_j(x)}$, we can assume that $m=0$ as long as the constants we find depends only on $\vert f_j\vert$ and not $f_j$.

\begin{multline*}
    \left\vert \sum_{j=1}^n  f_j(x)e^{itu_j(x)}-f_j(y)e^{itu_{j}(y)} \right\vert^2 \\
    \begin{aligned} &=  \left( \sum_{j=1}^n  f_j(x)e^{itu_j(x)}-f_j(y)e^{itu_{j}(y)} \right)\left( \sum_{k=1}^n  \overline{f_k(x)}e^{-itu_k(x)}-\overline{f_k(y)}e^{-itu_{k}(y)} \right) \\
     & = \sum_{j,k=1}^n f_j(x)\overline{f_k(x)}e^{it(u_j(x)-u_k(x))}-f_j(x)\overline{f_k(y)}e^{it(u_j(x)-u_k(y))}
      \\ &\phantom{==}-f_j(y)\overline{f_k(x)}e^{it(u_j(y)-u_k(x))}+f_j(y)\overline{f_k(y)}e^{it(u_j(y)-u_k(y))}.
      \end{aligned}
\end{multline*}

Note that if $z\in \R\setminus 2\pi\Z$, $\left\vert\sum_{t=0}^{N-1} e^{itz}\right\vert \leq \frac{1}{\vert \sin(z/2)\vert}$. The set $H=\bigcup_{j\neq k} \ker(u_j-u_k)$ is a finite union of hyperplane, thus $U'=U\cap (E\setminus H)$ is open and non-empty with $0$ in its closure.

Let $x$ in $U'$ with $\Vert x\Vert \leq \underset{j,k}{\min}\frac{\pi}{\Vert u_j\Vert+\Vert u_k\Vert}$ and such that there is $j_0$ with $f_{j_0}(x)\neq 0$. Then there is a neighbourhood $V_x$ of $x$ in $U'$ and $\varepsilon>0$ such that for $y\in V_x$ and $j\neq k$, $\varepsilon\leq \left\vert \frac{u_j(y)-u_k(y)}{2}\right\vert\leq \pi-\varepsilon$ and $\varepsilon\leq \left\vert \frac{u_j(x)-u_k(y)}{2}\right\vert\leq \pi-\varepsilon$.

Then, let $V_0=V_x-x$, there is $h_0\in V_0\setminus \bigcup_{i} \ker u_i$. Let $\eta=\frac{1}{2}\underset{j}{\min}\vert u_j(h_0)\vert>0$. Let $$V=V_0\setminus \{h\in E \vert \forall 1\leq j\leq n, \vert u_j(h)\vert\leq \eta\Vert h\Vert\}.$$ Then $V$ is an open subset of $U$, containing $\R^*h_0$ thus such that $0\in \overline{V}$. For any $h\in V$, by definition we have $\Vert h\Vert < \frac{1}{\eta}u_j(h)$ for any $1\leq j\leq n$.

Hence we get for any $y=x+h,h\in V$, $N\in \N$, \begin{multline}
   \frac{1}{N}\sum_{t=0}^{N-1}\left\vert \sum_{j=1}^n  f_j(x)e^{itu_j(x)}-f_j(y)e^{itu_{j}(y)} \right\vert^2  \geq  \\\sum_{j=1}^n \left(\vert f_j(x)\vert^2+\vert f_j(y)\vert^2 - \frac{\vert f_j(x)f_j(y)\vert}{N \left\vert \sin \frac{u_j(x)-u_j(y)}{2}\right\vert}\right)\\
      - \sum_{j\neq k} \left(\frac{\vert f_j(x)f_k(x)\vert}{N \left\vert \sin \frac{u_j(x)-u_k(x)}{2}\right\vert}+ \frac{\vert f_j(y)f_k(y)\vert}{N \left\vert \sin \frac{u_j(y)-u_k(y)}{2}\right\vert}+2 \frac{\vert f_j(x)f_k(y)\vert}{N \left\vert \sin \frac{u_j(x)-u_k(y)}{2}\right\vert}\right).
\end{multline}Now, for each of the terms with $j\neq k$, the assumptions on $x,y$ ensures that the arguments in $\sin$ are bounded away from $0,\pi$. Furthermore, up to restricting $V$ to a bounded set if necessary, the functions $f_i$ are bounded. Thus, there is $N_0$ such that for $N>N_0$, we get $$ \displaystyle\frac{1}{N}\sum_{t=0}^{N-1}\left\vert \sum_{j=1}^n  f_j(x)e^{itu_j(x)}-f_j(y)e^{itu_{j}(y)} \right\vert^2 \geq \frac{\vert f_{j_0}(x)\vert^2}{2}- \sum_{j=1}^n \frac{\vert f_j(x)f_j(x+h)\vert}{N \left\vert \sin \frac{u_j(h)}{2}\right\vert}.$$Finally, for each $j$, there is $d_j$ such that for any $h\in V$, $N\geq \frac{d_j}{\vert u_j(h)\vert}$, $$\frac{\vert f_j(x)f_j(x+h)\vert}{N \left\vert \sin \frac{u_j(h)}{2}\right\vert} \leq \frac{\vert f_{j_0}(x)\vert^2}{4n}.$$Thus, for $d>(\max d_j)/\eta$, then for any $y=x+h,h\in V$ and $N>\max\left(N_0,\frac{d}{\Vert h\Vert}\right)$, \begin{equation*}\frac{1}{N}\sum_{t=0}^{N-1}\left\vert \sum_{j=1}^n  f_j(x)e^{itu_j(x)}-f_j(y)e^{itu_{j}(y)} \right\vert^2 \geq \frac{\vert f_{j_0}(x)\vert^2}{4}.\qedhere\end{equation*}
\end{proof}

\begin{proof}[Proof of Theorem \ref{thm:opti}]
Consider $\lambda\in \mathfrak{a}^*$ such that $n(\lambda)=2\kappa(G)$, and such that $\langle\alpha,\lambda\rangle \geq 0$ for any $\alpha\in \Sigma^+$. By Proposition \ref{prop:defpos}, $\varphi_{t\lambda}$ is positive definite for any $t\in \R$. As in \eqref{eq:derivphient}, for any $Y\in \mathfrak{a}$, $t\geq 1$, $X\in \mathfrak{a}^r$, \begin{equation}
    \label{eq:derivphicasopti} D^r\psi_{t\lambda}(Y)(X)=\sum_{j=0}^r t^j \int_K e^{it\lambda(H(\exp(Y)k))} g_j(Y,X)(k)dk.
\end{equation}
Let $I_j(Y,X,t)=\int_K e^{it\lambda(H(\exp(Y)k)} g_j(Y,X)(k)dk$. Let $W_\lambda$ denote the stabiliser of $\lambda$ under the action of the Weyl group $W$ and $K_Y,K_\lambda$ be the centralisers of $Y,\lambda$ in $K$. Let also $$\Sigma^+(\lambda)=\{\alpha \in \Sigma^+ \vert \langle \alpha,\lambda\rangle\neq 0\}$$ and $$\sigma_w=-\underset{\alpha\in \Sigma^+(\lambda)\neq 0}{\sum} m(\alpha)\operatorname{sgn}(\langle\alpha,\lambda\rangle)(w\alpha)(Y)).$$Let $d_0k$ denote the Riemannian measure on $K$ induced by the bi-invariant metric defined by the Killing form on $\mathfrak{k}$. Let $\operatorname{Vol}(K)=\int_K d_0k$. We also denote by $d_0k$ the induced Riemannian measure on the submanifold $K_awK_\lambda$ - the measure coming from the restriction of the Riemannian metric of $K$ to a Riemannian metric on the submanifold. For $w\in W$, $g\in C^\infty(K)$, $Y\in \mathfrak{a}^+$, set \begin{equation}
    \label{eq:distrib} c_{w,a}(g)=e^{i\frac{\pi}{4}\sigma_w} \prod_{\alpha\in \Sigma^{+}(\lambda)}\left\vert\frac{\langle \alpha,\lambda \rangle}{4\pi}\left(1-e^{-2(w\alpha)(Y)}\right)^{-\frac{m(\alpha)}{2}}\right\vert \frac{1}{\operatorname{Vol}(K)}\int_{K_awK_{\lambda}}g(k)d_0k.
\end{equation}
Then, by \cite[Thm. 9.1]{DKV}, for any $Y\in \mathfrak{a}^+$, there is a neighbourhood $U_Y$ of $Y$ in $\mathfrak{a}^+$ and $D(Y)>0$ such that for any $0\leq j\leq r$ $t\geq 1$, $Y'\in U_Y$ and $X$ with $\Vert X_i\Vert=1$ for all $i$, \begin{equation}
    \label{eq:phasestatopti}\left\vert I_j(Y',X,t) - \sum_{W/W_\lambda} e^{it(w\lambda)(Y')}t^{-\kappa(G)}c_{w,Y'}(g_j(Y',X))\right\vert \leq D(Y)t^{-{\kappa(G)}-1}.
\end{equation}We use that $g_j$ is smooth in all variables hence bounded on compacts and that the bound is uniform in the parameter $Y'$ of the phase function.

In particular, combining \eqref{eq:derivphicasopti} with \eqref{eq:phasestatopti} for $0\leq j< r$, for any $Y$ there is a neighbourhood $V_Y$ of $Y$ and a constant $C(Y)$ such that for any $t\geq 1$, $Y'\in V_Y$ and $X$ with $\Vert X_i\Vert=1$, \begin{equation}\label{eq:step1}
    \left\vert D^r\psi_{t\lambda}(Y')(X)-t^r I_r(Y',X,t) \right\vert \leq C(Y)t^{-1}.
\end{equation}

For $X$ fixed with $\Vert X_i\Vert=1$, let $S_t(x)=\sum_{W/W_\lambda} e^{it(w\lambda)(x)}c_{w,x}(g_r(x,X))$. Combining \eqref{eq:step1} and \eqref{eq:phasestatopti}, if $t\geq 1$, and $x,y\in U_Y\cap V_Y$, \begin{equation}\label{eq:step2}\begin{aligned}
    t^{-\delta}\vert S_t(x)-S_t(y)\vert & \leq  t^r\vert t^{-\kappa(G)}S_t(x)-I_r(x,X,t)\vert + t^r\vert I_r(x,X,t)-I_r(y,X,t)\vert \\ 
    &\phantom{\leq}+ t^r\vert I_r(y,X,t)-t^{-\kappa(G)}S_t(y)\vert  \\
     & \leq  2D(Y)t^{-\delta-1}+\vert t^rI_r(x,H,t)-D^r\psi_{t\lambda}(x)(X) \vert \\
     &\phantom{\leq} + \vert D^r\psi_{t\lambda}(x)(X)-D^r\psi_{t\lambda}(y)(X) \vert\\
     &\phantom{\leq}+\vert D^r\psi_{t\lambda}(y)(X)-t^rI_r(y,X,t) \vert\\
     &\leq  2D(Y)t^{-\delta-1}+2C(Y)t^{-1} + \vert D^r\psi_{t\lambda}(x)(X)-D^r\psi_{t\lambda}(y)(X) \vert\\
     & \leq  \Vert D^r\psi_{t\lambda}(x)-D^r\psi_{t\lambda}(y) \Vert + 2(C(Y)+D(Y))t^{-1}.
\end{aligned}\end{equation}Now the functions $c_{w,x}(g_r(x,X))$ are all zero at $x$ if and only if $g_r(x,X)=0$ almost everywhere on $\bigcup K_awK_\lambda$. Let $f_k(x)=H(\exp(x)k)$, then $$g_r(x,X)(k)=e^{-\rho(H(\exp(x)k))}\prod_{i=1}^r \lambda(Df_k(x)(X_i)).$$ If $P_\mathfrak{a}$ denote the orthogonal projection onto $\mathfrak{a}$, by \cite[Section 5]{DKV} we have $$Df_k(x)(X_i)=P_\mathfrak{a}(\Ad(t(\exp(x)k)^{-1})(X_i))$$where $t(g)=a(g)n(g)$ in the Iwasawa decomposition (see \eqref{eq:iwasawa}). The function $k\mapsto f_k(x)$ is left $K_x$-invariant, and by \cite[Prop. 5.6]{DKV}, $k\mapsto \lambda(f_k(x))$ is right $K_\lambda$-invariant as $H_\lambda\in \overline{\mathfrak{a}^+}$. Thus, $g_r(x,X)$ is constant on $K_xwK_\lambda$ for any $w$. Thus given any open subset of $\mathfrak{a}^+$, we can choose $x,X$ such that $g_r(x,X)(e)\neq 0$. Thus the hypotheses of Lemma \ref{lem:expo} hold for the family of functions $f_w:x\mapsto c_{w,x}(g_r(x,X))$, for $U=\mathfrak{a}^+$. Let $C,d,x,V$ be given by Lemma \ref{lem:expo}, $W_x=x+V$ such that for any $y\in W_x$, $m\in \N$, $N\geq \frac{d}{\Vert x-y\Vert}$, \begin{equation}
    \label{eq:step3} \sum_{t=m}^{m+N-1} \vert S_t(x)-S_t(y)\vert^2 \geq CN.
\end{equation} From now on, we choose $Y=x$ given above. Let $M=4(C(x)+D(x))^2$, we get from \eqref{eq:step2} that for any $t\geq 1$, $y\in U_x\cap V_x\cap W_x$, \begin{equation}
    \label{eq:step4}\frac{t^{-2\delta}}{2}\vert S_t(x)-S_t(y)\vert^2 \leq \Vert D^r\psi_{t\lambda}(x)-D^r\psi_{t\lambda}(y) \Vert^2+ Mt^{-2}
\end{equation}

Assume now that the family of positive definite spherical functions of $(G,K)$ is bounded in $C^{(r,\delta')}(\mathfrak{a}^+)$ for $\delta'> \delta$. In particular, up to reducing $U_x\cap V_x\cap W_x$ to a bounded subset of diameter $L$ if necessary, there is $D>0$ such that for any $y\in U_x\cap V_x\cap W_x$ and $t\geq 1$, \begin{equation}
    \label{eq:hypo} \Vert D^r\psi_{t\lambda}(x)-D^r\psi_{t\lambda}(y) \Vert \leq D\Vert x-y\Vert^{\delta'}
\end{equation}

For $y$ fixed, set $m,N$ such that \begin{equation}\label{eq:choicem}\frac{1}{\Vert x-y\Vert^{\delta'}}\leq m\leq \frac{1}{\Vert x-y\Vert^{\delta'}}+1\end{equation}and \begin{equation}\label{eq:choiceN}\frac{d}{\Vert x-y\Vert}\leq N\leq \frac{d}{\Vert x-y\Vert}+1.\end{equation}Combining \eqref{eq:step3}, \eqref{eq:step4} and \eqref{eq:hypo} gives \begin{equation}
    \label{eq:step5} \begin{aligned}\frac{CN}{2(m+N)^{2\delta}}&\leq\sum_{t=m}^{m+N-1} \frac{t^{-2\delta}}{2} \vert S_t(x)-S_t(y)\vert^2 \\&\leq \sum_{t=m}^{m+N-1} \left(\Vert D^r\psi_{t\lambda}(x)-D^r\psi_{t\lambda}(y) \Vert^2+ Mt^{-2}\right)\\ &\leq ND^2\Vert x-y\Vert ^{2\delta'}+\frac{MN}{m^2}\end{aligned}
\end{equation}
thus \begin{equation}
    \label{eq:step6} \frac{C}{2(m+N)^{2\delta}} \leq D^2\Vert x-y\Vert ^{2\delta'}+\frac{M}{m^2} \leq (D^2+M)\Vert x-y\Vert^{2\delta'}
\end{equation}by \eqref{eq:choicem}. But by \eqref{eq:choicem} and \eqref{eq:choiceN}, we have \begin{multline}
    m+N\leq \frac{d}{\Vert x-y\Vert}+1+\frac{1}{\Vert x-y\Vert^{\delta'}}+1 \leq \frac{1}{\Vert x-y\Vert}\left(d+2\Vert x-y\Vert +\Vert x-y\Vert^{1-\delta'}\right)\\ \leq \frac{1}{\Vert x-y\Vert}\left(d+2L+L^{1-\delta'}\right)
\end{multline}hence \eqref{eq:step6} becomes \begin{equation}
    \label{eq:step7} \frac{C}{2(d+2L+L^{1-\delta'})^{2\delta}}\Vert x-y\Vert^{2\delta}\leq (D^2+M)\Vert x-y\Vert^{2\delta'}.
\end{equation}
Since \eqref{eq:step7} holds for any $y\in U_x\cap V_x\cap W_x$ with the constant involved independent from $y$ and $\delta'>\delta$, we get a contradiction as $y$ goes to $x$ (which is possible because $0\in \overline{V}$ hence $x\in \overline{U_x\cap V_x\cap W_x}$).
\end{proof}

\begin{coro}\label{coro:regoptKfini}Let $G$ be a connected semisimple Lie group with finite center and $K$ a maximal compact subgroup. Let $r=\lfloor \kappa(G)\rfloor$, $\delta=\kappa(G)-r$. Then any $K$-finite matrix coefficient of a unitary representation of $G$ is in $C^{(r,\delta)}(G_r)$. Furthermore, for any $\delta'>\delta$, there exists a $K$-bi-invariant matrix coefficient of $G$ that is not in $C^{(r,\delta')}(G_r)$.
\end{coro}
\begin{proof}
It follows from Lemma \ref{lem:lienspheriquekbiinv}, Corollary \ref{coro:liealgvsliegroup} and Theorems \ref{thm:kfinite}, \ref{thm:reg} and \ref{thm:opti}.
\end{proof}

\begin{remark}\label{rmk:singular}
For any open subset $U$ strictly larger than $G_r$, there are $K$-bi-invariant matrix coefficients that are only continuous. Indeed, for any $a=\exp(Y)\in U\setminus G_r$, there exists $\lambda\in \mathfrak{a}^*$ nonzero, $w\in W$ such that $\Sigma_w(\lambda,a)=\{\alpha\in \Sigma^+\vert \langle\alpha,\lambda\rangle\alpha(Y)\neq 0\}$ is empty.
Set $n_w=\sum_{\alpha\in \Sigma_w(\lambda,a)} m(\alpha)$ and $g_{Y'}:k\mapsto e^{-\rho(H(\exp(Y')k))}$, then by \cite[Thm. 9.1]{DKV}, $$\left\vert \varphi_{t\lambda}(\exp Y')-\sum_{w\in W/W_\lambda}e^{itw\lambda(Y')}t^{-n_w/2}c_{w,Y'}(g_{Y'})\right\vert\leq Ct^{-1}$$using the same notations as in Theorem \ref{thm:opti}. Since $n_w=0$ for some $w$, the same proof as in Theorem \ref{thm:opti} gives that the positive definite spherical functions are not bounded in any Hölder spaces.
\end{remark}

\section{Compact semisimple groups}\label{sec:cpt}
\subsection{An upper bound on regularity}\label{sec:cptupper}
We first recall some notations of Section \ref{sec:sphericalfunctions} and introduce new ones (more details in \cite{clerc, helgason1979differential}). If $\mathfrak{g}$ is a semisimple real Lie algebra, we introduced a decomposition $\mathfrak{g}=\mathfrak{k}\oplus \mathfrak{p}$ into eigenspaces of a Cartan involution $\theta$. Let $G_\C$ be the simply connected Lie group whose Lie algebra is the complexification $\mathfrak{g}_\C$ of $\mathfrak{g}$. Let $G,K$ be the analytic subgroups of $G_\C$ corresponding to the subalgebras $\mathfrak{g},\mathfrak{k}$. Consider $\mathfrak{u}=\mathfrak{k}\oplus i\mathfrak{p}$ and $U$ the corresponding analytic subgroup of $G_\C$. Then $U$ is a maximal compact subgroup of $G_\C$ and is simply connected. Consider also $K_\C$ the analytic subgroup corresponding to $\mathfrak{k}_\C$. Finally, recall that $\mathfrak{a}$ is a maximal abelian subspace of $\mathfrak{p}$ and $\mathfrak{n}=\bigoplus_{\alpha\in \Sigma^+} \mathfrak{g}^\alpha$. Let $A,A_\C,N,N_\C$ be the analytic subgroups of $G_\C$ corresponding to $\mathfrak{a},\mathfrak{a}_\C,\mathfrak{n},\mathfrak{n}_\C$.

The involution $\theta$ extends to $\mathfrak{g}_\C$ and thus induces an involution of $G_\C$ , also denoted $\theta$. The subgroup $K$ is the subgroup of fixed points of $\theta$ in $U$, so $(U,K)$ is a symmetric Gelfand pair and the symmetric space $M=U/K$ is the compact dual of $G/K$. Since the Killing form of $\mathfrak{g}_\C$ restricted to $\mathfrak{g}\times \mathfrak{g}$ coincides with the Killing form of $\mathfrak{g}$, we continue to denote $\langle \cdot,\cdot\rangle$ both on $\mathfrak{g}$ and its complexification. It must be noted that $\langle \cdot,\cdot\rangle$ is $\C$-bilinear on $\mathfrak{g}_\C$ and not sesquilinear, and thus not a scalar product.

Such pairs $(U,K)$ were studied in \cite{dumas2023regularity}. The optimal regularity of $K$-finite coefficients of $U$ was found in some specific cases and a conjecture was given in the general case. In what follows, we will extend the results using methods similar to what we did above in the non-compact setting.

Let $Q$ be the connected component of $\mathfrak{a}_r=\{H\in \mathfrak{a} \vert \forall \alpha\in \Sigma, \alpha(H)\not\in \pi \Z\}$ contained in $\mathfrak{a}^+$ and whose closure contains $0$. Then, there is a $KAK$ decomposition in the group $U$ (\cite[Prop. 5.8]{dumas2023regularity}).

\begin{prop}
\label{prop:kakcpt}For any $u\in U$, there exists a decomposition $$u=k_1(u)\exp (iP(u))k_2(u)^{-1}$$where $k_1(u),k_2(u)\in K$ and $P(u)\in Q$. The map $u\mapsto P(u)$ is smooth on the set $U_r=K\exp(iQ)K$. Furthermore, for each $u\in U_r$, there exists a neighbourhood $V_u$ of $u$ in $U_r$ and a choice of $u\mapsto k_i(u)$ such that $k_i$ is smooth on $V_u$, $i=1,2$.
\end{prop}

The set $U_r$ is a dense open subset of $U$ and we call it the set of regular points. This set will play the same role as $G_r$ in the non-compact case.

As in the previous section for the non-compact setting, we want to study the spherical functions of the pair $(U,K)$. Let $\widehat{U}_K$ denote the set of classes of irreducible finite dimensional representations of $U$ with a non-zero $K$-invariant vector. Then $\widehat{U}_K$ (and thus spherical functions of $(U,K)$) are parameterized by a subset of $\mathfrak{a}_\C^*$. Note that since $U$ is compact, any spherical function is positive-definite (\cite[Thm. 6.5.1]{Dijk+2009}).
\begin{theorem}[Cartan-Helgason]
\label{cartanhelgason}Let $\Lambda=\{\mu\in \mathfrak{a}^* \vert \forall \alpha\in \Sigma_\mathfrak{a}^+,\frac{\langle\mu,\alpha\rangle}{\langle\alpha,\alpha\rangle}\in \N\}$. Then the map which sends a representation to its highest weight is a bijection from $\widehat{U}_K$ onto $\Lambda$.
\end{theorem}

There exists $\mu_1,\cdots,\mu_\ell\in \mathfrak{a}^*$ such that$$\frac{\langle\mu_i,\alpha_j\rangle}{\langle\alpha_j,\alpha_j\rangle}=\left\lbrace \begin{aligned}0 & \textrm{ if }i\neq j\\ 1& \textrm{ if }i=j,2\alpha_j\not\in \Sigma_\mathfrak{a}^+\\ 2& \textrm{ if }i=j,2\alpha_j\in \Sigma_\mathfrak{a}^+\end{aligned}\right.$$The elements $(\mu_i)$ are called fundamental weights and $\Lambda=\left\{\sum m_i\mu_i, m_i\in \N\right\}$ (see \cite{vretare_orth}).

Let $\mu\in\Lambda$ and $\pi_\mu$ an irreducible finite-dimensional representation of $U$ with highest weight $\mu$. Let $e_K$ be a unit $K$-invariant vector. Then $\psi_\mu:u\mapsto \langle \pi_\mu(u) e_K,e_K\rangle$ is a spherical function of $(U,K)$. Since $\pi_\mu$ is a finite-dimensional representation, it is smooth and its differential induces a representation of $\mathfrak{u}$, which extends to $\mathfrak{u}_\C=\mathfrak{g}_\C$ and is itself the differential of a representation of the simply connected group $G_\C$. Thus, $\pi_\mu$ extends to a representation of $G_\C$, so $\psi_\mu$ is defined on all of $G_\C$. Given the notations of Section \ref{sec:sphericalfunctions}, $\psi_\mu\vert_G=\varphi_{-i(\mu+\rho)}$.

Thus on $G$, we know that $\psi_\mu$ has an integral representation. But since the Iwasawa decomposition does not extend to $G_\C$, the integral does not have a meaning outside of $G$. However, even if $K_\C \times A_\C \times N_\C\to G_\C$ is not a diffeomorphism, it is a diffeomorphism in a neighbourhood of the identity $e\in G_\C$ so we can still work there. The following lemma is \cite[Lemme 1]{Clers1976}.

\begin{lem}\label{lem:extan}
    There exists a neighbourhood $V$ of $e$ in $G_\C$ which is invariant by conjugation by $K$ and analytic maps $\kappa:V\to K_\C$, $n:V\to N_\C$ and $H:V\to \mathfrak{a}_\C$ such that\begin{enumerate}
        \item $H(e)=0$,
        \item $\forall g\in V$, $g= \kappa(g)\exp H(g)n(g)$.
    \end{enumerate}
\end{lem}
The map $H$ coincides with the Iwasawa projection on $G\cap V$. Since $H(k^{-1}gk)=H(gk)$ for any $g\in G$, $k\in K$, we can extend the expression of $\psi_\mu$ on $G\cap V$ to all of $V$ by analytic continuation (\cite[Lemme 3]{Clers1976}).

\begin{lem}\label{lem:spheriqueext}
    Let $\mu\in \Lambda$. For any $g\in V$, $\psi_\mu(g)=\int_K e^{\mu(H(k^{-1}gk))}dk$.
\end{lem}

Unlike the non-compact case, the phase function is now complex-valued. However, for any $u\in U\cap V$, $\re\mu(H(k^{-1}uk))\leq 0$ (\cite[Coro. 2.4]{clerc}) which is the condition to apply the method of stationary phase (see \cite{melin}).\medskip

We will now state the version of the stationary phase approximation we will use, from \cite{Chazarain1974} and \cite{melin}.
\begin{theorem}\label{thm:statphase}
    Let $(Z,g)$ be a compact Riemannian manifold of dimension $d$, $dz$ its volume measure and $U$ an open subset of $\R^n$. Let $f\in C^\infty(Z)$ and $\phi\in C^\infty(Z\times U)$ be complex-valued functions. Let $W_a$ be the set of critical points of $\phi_a:z\mapsto\phi(z,a)$ for $a\in U$ and assume that $W_a=W$ for any $a$. Assume also that $W$ is finite, and for any $w\in W$, the Hessian of $\phi_a$ at $w$ is non-degenerate. Furthermore, suppose that $\re\phi\leq 0$, with equality at critical points $w\in W$. Set $$I(f,a,t)=\int_Z e^{t\phi(z,a)}f(z)dz$$ and fix $a_0\in U$. Then there exists a semi-norm $\nu$ on $C^\infty(Z)$ and an open neighbourhood $U'\subset U$ of $a_0$ such that for any $t\geq 1$, $a\in U'$, $f\in C^\infty(Z)$, $$\left\vert I(f,a,t)-\sum_{w\in W}e^{t\phi(w,a)}t^{-d/2} f(w)\left(\frac{(2\pi)^d}{\det (-\Hess_{\phi_a}(w))}\right)^{\frac{1}{2}}\right\vert\leq v(f)t^{-\frac{d}{2}-1}$$where the square root is taken as the branch of the square root which is deformed to $1$ under the homotopy $(1-s)(-\Hess_{\phi_a}(w)))+s\Id$.
\end{theorem}
\begin{proof}
    We can cover $Z$ by a finite number of chart open subsets $Z_j$, $j\in J$ which contains at most one element of $W$. We may assume that $\Vert D\phi_a(z) \Vert$ is bounded below on $Z_j$ which does not contain a critical point. Using a partition of unity subordinated to this open cover, we can write $$I(f,a,t)=\sum_{j\in J} I_j(f,a,t)$$where $$I_j(f,a,t)=\int_{Z_j} e^{t\phi(z,a)}f_j(z)dz$$and $\sum f_j(z)=f(z)$ for any $z\in Z$.
    In the chart $Z_j$, consider the local coordinates given by $H_j:Z_j\to \R^d$, chosen such that if $w\in W\cap Z_j$, $H_j(w)=0$. Let $G=\det(g_{i,j})$ where $g_{ij}(z)=g_z(\partial_i,\partial_j)$. Then by definition of the volume measure on $Z$, we have $$I_j(f,a,t)=\int_{\R^d} e^{\phi(H_j^{-1})(x),a}f_j(H_j^{-1}(x))\sqrt{G(H_j^{-1}(x))}dx.$$
    
    First, if $Z_j$ has no critical points, by \cite[Thm. 7.7.1]{hormander1983analysis}, for any $n>0$, there exists $C,C'>0$ such that $$\vert I_j(f,a,t)\vert \leq \frac{C}{t^n}\sum_{\vert k\vert \leq n} \sup \Vert D^k(f_j\sqrt{G})(x)\Vert\leq  \frac{C'}{t^n}\sum_{\vert k\vert \leq n} \sup \Vert D^k(f)(x)\Vert.$$We used Leibniz formula to replace $f_j\sqrt{G}$ by $f$, up to changing the constant $C$ to some $C'$ taking into account norms of differentials of $G$ and of the partition of unity.

    If $w\in Z_j$ is critical, by \cite[Thm. 2.3]{melin}, \cite[Thm. 7.7.5]{hormander1983analysis}, there exists an open neighbourhood $U_j\subset U$ of $a_0$ such that for any $a_0\in U_j$, $$\left\vert I_j(f,a,t)-e^{t\phi_a(w)}t^{-d/2} f_j(w)\sqrt{G(w)}\left(\frac{(2\pi)^d}{\det (-\Hess_{\phi_a\circ H_j^{-1}}(0))}\right)^{\frac{1}{2}}\right\vert\leq v_j(f)t^{-\frac{d}{2}-1}$$where the square root is taken as in the statement of the theorem. Again, the semi-norm $v_j$ should be applied to $f_j\sqrt{G}$ instead of $f$, but since it is defined as a differential operator, by Leibniz formula the inequality remains true with $f$ up to changing the semi-norm.

    For $u,v\in T_0\R^d$, at the critical point we have that $$\Hess_{\phi_a\circ H_j^{-1}}(0)(u,v)=\Hess_{\phi_a}(w)(DH_j^{-1}(0)u,DH_j^{-1}(0)v).$$Fix an orthonormal basis of $T_wZ$ with respect to the inner product $g_w$ and consider the canonical basis $\frac{\partial}{\partial x_i}\vert_0$ of $T_0\R^d$, then $$\det (-\Hess_{\phi_a\circ H_j^{-1}}(0))=\det (-\Hess_{\phi_a}(w))\left(\det DH_j^{-1}(0)\right)^2.$$But since $\partial_i\vert_w=DH_j^{-1}(0) \left(\frac{\partial}{\partial x_i}\vert_0\right)$, we also have that $(g_{ij}(w))_{1\leq i,j\leq n}={}^tAA$ where $A$ is the matrix of $DH_j^{-1}(0)$ in the previous bases, thus $G(w)=\det(DH_j^{-1}(0))^2$. Thus, the previous inequality becomes $$\left\vert I_j(f,a,t)-e^{t\phi_a(w)}t^{-d/2} f_j(w)\left(\frac{(2\pi)^d}{\det (-\Hess_{\phi_a}(w))}\right)^{\frac{1}{2}}\right\vert\leq v_j(f)t^{-\frac{d}{2}-1}.$$

    Thus by triangular inequality, setting $U'=\bigcap U_j$, there is a semi-norm $\nu$ on $C^\infty(Z)$ such that for any $f\in C^\infty(Z)$, $a\in U'$, and $t\geq 1$, $$\left\vert I(f,a,t)-\sum_{w\in W}e^{t\phi(w,a)}t^{-d/2} f(w)\left(\frac{(2\pi)^d}{\det (-\Hess_{\phi_a}(w))}\right)^{\frac{1}{2}}\right\vert\leq v(f)t^{-\frac{d}{2}-1}$$ which concludes the proof.
\end{proof}

We now compute the critical points and the Hessian of the phase function to make use of the previous theorem. This is mostly an application of the results of \cite{DKV} on $G$ that we used in Section \ref{sec:mainsec} and analytic continuation arguments.

For $a\in V$, consider the phase function $$\fonction{F_{a,\mu}}{K}{\C}{k}{\mu(H(k^{-1}ak))}.$$For any $k$, the map $a\mapsto F_{a,\mu}(k)$ is an analytic continuation of the phase function studied in \cite[Section 4]{DKV}. Denote $H_\mu\in \mathfrak{a}$ the unique vector such that for any $H\in \mathfrak{a}_\C$, $\mu(H)=\langle H,H_\mu\rangle$. Recall that $K_\mu$ is the centraliser in $K$ of $H_\mu$. For $\mu\in \Lambda$, $H_\mu \in \overline{\mathfrak{a}^+}$. Then, by \cite[Proposition 5.6]{DKV} and analytic continuation, the map $F_{a,\mu}$ is right $K_\mu$-invariant.

For $x\in G$, let $$\fonction{\theta_x}{K}{K}{k}{\kappa(xk)}.$$By uniqueness of the Iwasawa decomposition on $G$, it is clear that for any $x\in G$ and $k\in K$, $\kappa(xk)=k\kappa(k^{-1}xk)$. Now for $x\in V$, since we can extend $\kappa$ analytically on $V$ by Lemma \ref{lem:extan} and $V$ is $K$-invariant, the element $k\kappa(k^{-1}xk)\in K_\C$ is well-defined and the formula $$\fonction{\theta_x}{K}{K_\C}{k}{k\kappa(k^{-1}xk)}$$extends $\theta$ on $V\times K$, and $x\mapsto \theta_x(k)$ is analytic on $V$ for each $k$. Let $a\in V\cap U_r$ such that $a=\exp(iY)$ with $Y\in Q$.
\begin{lem}\label{lem:crit}
    The set of critical points of $F_{a,\mu}$ is $$\mathcal{C}_\mu=\bigcup_{w\in W} k_wK_\mu$$where $W=N_K(\mathfrak{a})/Z_{K}(\mathfrak{a})$ and $k_w$ is a representative of $w\in W$.
\end{lem}

\begin{proof}
    We identify $T_kK$ with $\mathfrak{k}$ under the isomorphism $T_eL_k$. Then by \cite[Lemma 5.1 and Corollary 5.2]{DKV}, for any $x\in G$, the tangent map of $F_{x,\mu}$ at $k$ is \begin{equation}\label{eq:tangphasecpt}T_kF_{x,\mu}:Z\mapsto \langle Z, \Ad(n(k^{-1}xk)^{-1})(H_\mu)\rangle.\end{equation}Again by Lemma \ref{lem:extan}, for $k\in K$ and $Z\in \mathfrak{k}$ fixed, the right-hand side of \eqref{eq:tangphasecpt} extends to a well-defined analytic function of $x$ on $V$. The map $x\mapsto F_{x,\mu}(k)$ is also analytic on $V$ for $k\in K$ fixed. For $Z\in \mathfrak{k}$, $$T_kF_{x,\mu}(Z)=\underset{t\to 0}{\lim} \frac{F_{x,\mu}(k\exp(tZ))-F_{x,\mu}(k)}{t}$$and the convergence is uniform for $x$ in a compact set. Thus, as a uniform limit of analytic functions, $T_kF_{x,\mu}(Z)$ is analytic in $x\in V$. Thus, both sides of \eqref{eq:tangphasecpt} extends analytically to functions of $x\in V$ which coincides on $G\cap V$, thus by uniqueness of analytic continuation, \eqref{eq:tangphasecpt} holds for any $x\in V$.

    Similar arguments of analytic continuation will be used several times in what follows. We will not give full details as the proofs are identical.

    By \cite[Prop. 5.4]{DKV}, any $k\in C_\mu$ is a critical point of $F_{a,\mu}$ for $a\in A$. For $\mu\in \Lambda,k\in C_\mu,Z\in \mathfrak{k}$ fixed, the map $a\mapsto T_kF_{a,\mu}(Z)$ is analytic in $a\in V\cap A_\C$ and zero on $A$. Thus, by analytic continuation it is identically $0$ and $k$ is a critical point of $F_{a,\mu}$.

    Conversely, let $k$ be a critical point of $F_{a,\mu}$. Then $\Ad(n(k^{-1}ak)^{-1})(H_\mu)\in \mathfrak{g}_\C$ is orthogonal to $\mathfrak{k}$ so $\Ad(n(k^{-1}ak)^{-1})(H_\mu)\in \mathfrak{p}_\C$. Since $H_\mu\in \mathfrak{a}\subset \mathfrak{p}_\C$, we get that $$\Ad(n(k^{-1}ak)^{-1})(H_\mu)-H_\mu\in \mathfrak{p}_\C.$$Furthermore, for any $X\in \mathfrak{n}_\C$, $\Ad(\exp(X))(H_\mu)=e^{\ad(X)}(H_\mu)=H_\mu \mod \mathfrak{n}_\C$. Since $\exp(\mathfrak{n}_\C)$ generates $N_\C$, we get $\Ad(n)(H_\mu)=H_\mu \mod \mathfrak{n}_\C$ for any $n\in N_\C$. Thus, $$\Ad(n(k^{-1}ak)^{-1})(H_\mu)-H_\mu\in \mathfrak{n}_\C$$so $$\Ad(n(k^{-1}ak)^{-1})(H_\mu)-H_\mu\in \mathfrak{p}_\C \cap \mathfrak{n}_\C=\{0\}.$$So we get \begin{equation}
        \label{eq:nstab} \Ad(n(k^{-1}ak))(H_\mu)=H_\mu.
    \end{equation}

    From \cite[Lemma 1.1 and Lemma 5.9]{DKV}, we also have for $a'=\exp(Y')\in A$, $k\in K$ and $Z\in \mathfrak{k}$ that $$T_kF_{a',\mu}(Z)=-\langle [Y',\Ad(\theta_{a'}(k))(H_\mu)],\Tilde{Z}\rangle$$where $\Tilde{Z}= \left(\sinh \ad(Y')/\ad(Y')\right)\circ \Ad(k)(Z)$. By analytic continuation, we extend this expression to $V$. In particular for $a=\exp(iY)$, we get $$T_kF_{a,\mu}(Z)=-i\langle [Y,\Ad(\theta_a(k))(H_\mu)],\Tilde{Z}\rangle$$where $\Tilde{Z}= \left(\sin \ad(Y)/\ad(Y)\right)\circ \Ad(k)(Z)$. 
    
    We claim that the map $Z\mapsto \Tilde{Z}$ is an isomorphism of $\mathfrak{k}$. Since $\Ad(k)$ is an isomorphism of $\mathfrak{k}$, it suffices to show that $$T=(\sin \ad(Y)/\ad(Y))=\sum_{n\geq 0} \frac{(-1)^n}{(2n+1)!}\ad(Y)^{2n}$$is an isomorphism of $\mathfrak{k}$. Consider the basis of $\mathfrak{k}$ used in the proof of Lemma \ref{lem:kaksubmersion}. Let $Y_1,\cdots,Y_r$ be a basis of $\mathfrak{m}=\mathfrak{k}^{\mathfrak{a}}$. For $\alpha\in \Sigma^+$, let $Z_{\alpha,1},\cdots,Z_{\alpha,m(\alpha)}$ be a basis of $\mathfrak{g}^\alpha$. Let $Z_{\alpha,i}=Z_{\alpha,i}+\theta(Z_{\alpha,i})\in \mathfrak{k}$. Then $(Y_i)_{1\leq i\leq r}\cup (Z_{\alpha,i}^+)_{\alpha\in \Sigma^+,1\leq i\leq m(\alpha)}$ is a basis of $\mathfrak{k}$. In this basis, $T$ is a diagonal operator and its eigenvalues are nonzero. Indeed, for any $1\leq i\leq r$, we have $T(Y_i)=Y_i$ - since $Y_i$ commutes with $\mathfrak{a}$. For any $\alpha\in \Sigma^+$ and $1\leq i\leq m(\alpha)$, we have $\ad(Y)(Z_{\alpha,i})=\alpha(Y)Z_{\alpha,i}$ and $\ad(Y)\theta(Z_{\alpha,i})=-\alpha(Y)\theta(Z_{\alpha,i})$. Thus, $\ad(Y)^2(Z_{\alpha,i}^+)=\alpha(Y)^2Z_{\alpha,i}^+$. Hence, $$T(Z_{\alpha,i}^+)=\sum_{n\in \N} \frac{(-1)^n\alpha(Y)^{2n}}{(2n+1)!}Z_{\alpha,i}^+=\frac{\sin(\alpha(Y))}{\alpha(Y)}Z_{\alpha,i}.$$Since we assumed $Y\in Q$, $\alpha(Y)\not\in \pi\Z$ for any $\alpha\in \Sigma^+$ and thus the claim is proved and $T$ is an isomorphism.
    
    Hence if $k$ a critical point of $F_{a,\mu}$, we get that $[Y,\Ad(\theta_a(k))(H_\mu)]\in \mathfrak{g}_\C$ is orthogonal to $\mathfrak{k}$ thus $[Y,\Ad(\theta_a(k))(H_\mu)]\in \mathfrak{p}_\C$. But since $\theta_a(k)\in K_\C$, we also have $\Ad(\theta_a(k))(H_\mu)\in \mathfrak{p}_\C$ thus \begin{equation}\label{eq:kstab}[Y,\Ad(\theta_a(k))(H_\mu)]\in [\mathfrak{p}_\C,\mathfrak{p}_\C]\cap \mathfrak{p}_\C=\mathfrak{k}_\C\cap  \mathfrak{p}_\C=\{0\}.\end{equation}
    Since $k^{-1}ak\in V$, we have by definitions of $\kappa,H,n$ that $$ak=k\kappa(k^{-1}ak)\exp(H(k^{-1}ak))n(k^{-1}ak)=\theta_a(k)\exp(H(k^{-1}ak))n(k^{-1}ak).$$
    Thus combining \eqref{eq:nstab} and \eqref{eq:kstab}, we get that \begin{align*}
        [Y,\Ad(ak)(H_\mu)] & = [Y,\Ad(\theta_a(k))\circ \Ad(\exp(H(k^{-1}ak)))\circ \Ad(n(k^{-1}ak))(H_\mu)]\\
        & = [Y,\Ad(\theta_a(k)(H_\mu)]\\
        & = 0
    \end{align*}
    Since $\Ad(a)$ is an automorphism of Lie algebra and $A_\C$ is abelian, we get $$[\Ad(a^{-1})(Y),\Ad(k)(H_\mu)]=[Y,\Ad(k)(H_\mu)]=0.$$Now since $k\in K$ and $Y$ is regular, we conclude with \cite[Prop. 1.2]{DKV} that $k\in \mathcal{C}_\mu$.
\end{proof}

We now see $F_{a,\mu}$ as a function defined on $K/K_\mu$, whose distinct critical points are $k_wK_\mu$ for $w\in W/W_\mu$ - thus, there are only finitely many such points. For $\alpha\in \Sigma^+$, let $\mathfrak{k}^\alpha=\mathfrak{k}\cap (\mathfrak{g}^\alpha \oplus \mathfrak{g}^{-\alpha})$. Denote $\Sigma^+(\mu)=\left\lbrace \alpha\in \Sigma^+ \vert \langle \alpha,\mu\rangle \neq 0\right\rbrace$ and $$\mathfrak{l}_\mu=\bigoplus_{\alpha\in \Sigma^+(\mu)} \mathfrak{k}^\alpha.$$Then $\dim K/K_\mu=\dim(\mathfrak{l}_\mu)=\sum_{\alpha\in \Sigma^{+}(\mu)}m(\alpha)=n(\mu)$. Let $F_\alpha:\mathfrak{l}_\mu \to \mathfrak{k}^\alpha$ be the orthogonal projection.

Let $w\in W/W_\mu$. The value of the phase is $F_{a,\mu}(k_wK_\mu)=i\mu(w^{-1}Y)=i(w\mu)(Y)$. By analytic continuation and \cite[Prop. 6.5]{DKV}, the Hessian at $k_wK_\mu$ is given by $$\Hess_{F_{a,\mu}}(k_wK_\mu)(Y,Z)=-\langle Y,L_{a,\mu,w}(Z)\rangle$$where \begin{equation}\label{eq:hessvp}L_{a,\mu,w}=-\frac{1}{2}\sum_{\alpha\in \Sigma^+(\mu)}\langle\alpha,\mu\rangle\left(1-e^{-2(w\alpha)(iY)}\right)F_\alpha\end{equation}and $-\langle \cdot,\cdot\rangle$ is positive-definite on $\mathfrak{k}$. Finally, to apply the method of stationary phase (\cite{melin, Chazarain1974}), we need to compute the square root of $\det(-L_{a,\mu,w})^{-1}$ which is continuously deformed to $1$ by the homotopy $s\Id+(1-s)(-L_{a,\mu,w})$.

If $w\alpha\in \Sigma^+$, we have \begin{equation*}
    \frac{1}{2}\langle\alpha,\mu\rangle\left(1-e^{-2(w\alpha)(iY)}\right)  = i\langle\alpha,\mu\rangle e^{-i(w\alpha)(Y)}\sin((w\alpha)(Y))
\end{equation*}and $0<(w\alpha)(Y)<\pi$ so we take $e^{i\frac{\pi}{4}}\langle\alpha,\mu\rangle^{\frac{1}{2}} e^{-i\frac{(w\alpha)(Y)}{2}}\sin((w\alpha)(Y))^{\frac{1}{2}}$ as a square root.

If $w\alpha\in -\Sigma^+$, we have \begin{equation*}
    \frac{1}{2}\langle\alpha,\mu\rangle\left(1-e^{-2(w\alpha)(iY)}\right)  = -i\langle\alpha,\mu\rangle e^{-i(w\alpha)(Y)}\vert\sin((w\alpha)(Y))\vert
\end{equation*}since $-\pi<(w\alpha)(Y)<0$, so we take $e^{-i\frac{\pi}{4}}\langle\alpha,\mu\rangle^{\frac{1}{2}} e^{-i\frac{(w\alpha)(Y)}{2}}\vert\sin((w\alpha)(Y))\vert^{\frac{1}{2}}$ as a square root.

Let $\sigma_w=\sum_{\alpha\in \Sigma^{+}(\mu)}m(\alpha)\operatorname{sgn}((w\alpha)(Y))$. We get \begin{equation}
    \det(-L_{a,\mu,w})^{-\frac{1}{2}} = e^{-i\sigma_w \frac{\pi}{4}}e^{i(w\rho_\mu)(Y)}\prod_{\alpha\in \Sigma^+(\mu)} \langle\alpha,\mu\rangle^{-\frac{m(\alpha)}{2}} \vert\sin((w\alpha)(Y))\vert^{-\frac{m(\alpha)}{2}}
\end{equation}where $\rho_\mu=\frac{1}{2}\sum_{\alpha\in \Sigma^+(\mu)} m(\alpha)\alpha$.

\begin{theorem}\label{thm:regmaxcpt}
    For $\mu\in \Lambda$, let $\Psi_\mu = \psi_\mu \circ \exp\vert_{iQ}$. Let $r=\lfloor \kappa(G)\rfloor$ and $\delta=\kappa(G)-r$. Then for any $\delta'>\delta$, the family $\left(\Psi_\mu\right)_{\mu\in \Lambda}$ is not bounded in $C^{(r,\delta')}(iQ)$.
\end{theorem}
The proof of this result is a variation of the proof of Theorem \ref{thm:opti} using the expression of spherical functions of $(U,K)$ from Lemma \ref{lem:spheriqueext} and the method of stationary phase for complex-valued phase functions.
\begin{proof}
    By \eqref{kappamin}, we see that $\kappa(G)$ is attained on fundamental weights, which are elements of $\Lambda$. Thus, we can choose $\mu\in \Lambda$ be such that $n(\mu)=2\kappa(G)$. Let $V$ be as above given by Lemma \ref{lem:extan} and $V'$ an open subset of $Q$ such that $a=\exp(iY)\in V$ for any $Y\in V'$. Let $d(kK_\mu)$ be the image of the Haar measure on $K$ on $K/K_\mu$. Then for any $n\in \N$ and $Y\in V'$, $\Psi_{n\mu}(iY)=\int_{K/K_\mu} e^{n\mu(H(k^{-1}\exp(iY)k))}d(kK_\mu)$. As in the proof of Theorem \ref{thm:opti}, for any $X\in \mathfrak{a}^r$ there are functions $g_j(Y,X)\in C^\infty(K/K_\mu)$ which do not depend on $n\in \N$ such that $$D^r\Psi_{n\mu}(iY)(X)=\sum_{j=0}^r n^j \int_{K/K_\mu} e^{nF_{a,\mu}(k)}g_j(Y,X)(k)d(kK_\mu).$$
    
    Let $I_j(Y,X,n)=\int_{K/K_\mu} e^{nF_{a,\mu}(k)}g_j(Y,X)(k)d(kK_\mu)$ for $1\leq j\leq r$. We want to evaluate this integral as $n$ goes to infinity. Let $d_0(kK_\mu)$ be the volume measure on $K/K_\mu$ associated to the (invariant) Riemannian metric induced by the restriction of the inner product $-\langle\cdot,\cdot\rangle$ on $\mathfrak{k}$. Denote $\operatorname{Vol}(K/K_\mu)=\int_{K/K_\mu} d_0(K/K_\mu)$. By uniqueness of the invariant measure on $K/K_\mu$, we have $d(kK_\mu)=\frac{1}{\operatorname{Vol}(K/K_\mu)}d_0(kK_\mu)$. If $w\in W/W_\mu$, $g\in C^\infty(K/K_\mu)$, let $$c_{w,Y}(g)=\frac{(2\pi)^{n(\mu)/2}}{\operatorname{Vol(K/K_\mu)}}\det(-L_{\exp(iY),\mu,w})^{-\frac{1}{2}}g(k_wK_\mu).$$ By the method of stationary phase (Theorem \ref{thm:statphase}) and the computations on the phase functions $F_{a,\mu}$, there is a neighbourhood $U_Y$ of $Y$ and a constant $D(Y)\geq 0$ such that for any $1\leq j\leq r$, for any $n\in \N^*$, $Y'\in U_Y$ in $V'$ and $X\in \mathfrak{a}^r$ with $\Vert X_i\Vert=1$ for all $i$, \begin{equation}\label{eq:staphasecompact}
        \left\vert I_j(Y',X,n)-\sum_{w\in W/W_\mu} e^{in(w\mu)(Y')}n^{-\frac{n(\mu)}{2}}c_{w,Y'}(g_j(Y',X)) \right\vert \leq D(Y)n^{-\frac{n(\mu)}{2}-1}.
    \end{equation}
    
    We used that $g_j:V'\times \mathfrak{a}^r\times K\to \C$ is smooth, so that $\nu(g_j(Y',X)$ is bounded when $Y',X$ stay in compact sets. Using the previous inequality in the cases $1\leq j<r$, for any $Y\in V'$, there is a neighbourhood $V_Y$ of $Y$ and a constant $C(Y)>0$ such that for any $n\geq 1$, $Y'\in V_Y$ and $X$ with $\Vert X_i\Vert=1$ for all $1\leq i\leq r$, \begin{equation}
        \label{eq:distspheriqueint} \left\vert D^r\Psi_{n\mu}(iY')(X)- n^r I_r(Y',X,n)\right\vert \leq C(Y)n^{-1}.
    \end{equation}

    Let $S_n(x)=\sum_{w\in W/W_\mu} e^{in(w\mu)(x)}c_{w,x}(g_r(x,X))$. Since $n(\mu)=2\kappa(G)$ by choice of $\mu$, combining \eqref{eq:staphasecompact} with $j=r$ and \eqref{eq:distspheriqueint} yields that for any $n\in \N^*$ and $x,y\in U_Y\cap V_Y$, 

    \begin{equation}\label{eq:spheriquetolem}\begin{aligned}
    n^{-\delta}\vert S_n(x)-S_n(y)\vert & \leq  n^r\vert n^{-\kappa(G)}S_n(x)-I_r(x,X,n)\vert + n^r\vert I_r(x,X,n)-I_r(y,X,n)\vert \\ 
    &\phantom{\leq}+ n^r\vert I_r(y,X,n)-n^{-\kappa(G)}S_n(y)\vert  \\
     & \leq  2D(Y)n^{r-\kappa(G)-1}+\vert n^rI_r(x,H,n)-D^r\Psi_{n\lambda}(x)(X) \vert \\
     &\phantom{\leq} + \vert D^r\Psi_{n\lambda}(x)(X)-D^r\Psi_{n\lambda}(y)(X) \vert\\
     &\phantom{\leq}+\vert D^r\Psi_{n\lambda}(y)(X)-n^rI_r(y,X,n) \vert\\
     &\leq  2D(Y)n^{-\delta-1}+2C(Y)n^{-1} + \vert D^r\Psi_{n\lambda}(x)(X)-D^r\psi_{n\lambda}(y)(X) \vert\\
     & \leq  \Vert D^r\Psi_{n\lambda}(x)-D^r\Psi_{n\lambda}(y) \Vert + 2(C(Y)+D(Y))n^{-1}.
\end{aligned}\end{equation}

The hypotheses of Lemma \ref{lem:expo} hold for the family of functions $x\mapsto c_{w,x}(g_r(x,X))$, for $U=V'$. Let $C,d,x,V$ be given by Lemma \ref{lem:expo}, $W_x=x+V$ such that for any $y\in W_x$, $m\in \N$, $N\geq \frac{d}{\Vert x-y\Vert}$, \begin{equation}
    \label{eq:applem} \sum_{t=m}^{m+N-1} \vert S_n(x)-S_n(y)\vert^2 \geq CN.
\end{equation} 

The end of the proof follows as in Theorem \ref{thm:opti}.
\end{proof}

\begin{coro}\label{coro:cpt}
    Let $r=\lfloor \kappa(G)\rfloor$ and $\delta=\kappa(G)-r$. Then for any $\delta'>\delta$, there exists a $K$-bi-invariant matrix coefficient of a unitary representation of $U$ which is not in $C^{(r,\delta')}(U_r)$.
\end{coro}
\begin{proof}
    By Theorem \ref{thm:regmaxcpt}, the family of spherical functions viewed on the Lie algebra is not bounded in $C^{(r,\delta')}(Q)$. By Proposition \ref{prop:kakcpt} and Lemma \ref{lem:precomposition}, it follows that the family of spherical functions of $(U,K)$ is not bounded in $C^{(r,\delta')}(U_r)$. Thus, the result follows from Lemma \ref{lem:lienspheriquekbiinv}.\end{proof}

In \cite{dumas2023regularity}, we made a conjecture on the expected optimal regularity of coefficients of $(U,K)$, which should be $\kappa(G)$.
\begin{Conj}\label{conj}Let $r=\lfloor \kappa(G)\rfloor$ and $\delta=\kappa(G)-r$. Any $K$-finite matrix coefficient of a unitary representation of $U$ is in $C^{(r,\delta)}(U_r)$ and this regularity is optimal.
\end{Conj}

By \cite[Thm. 5.2]{dumas2023regularity}, it suffices to look at $K$-bi-invariant matrix coefficients. This conjecture has been proven in rank $1$ and when $U/K$ is a Lie isomorphic to a Lie group (which is equivalent to $G$ being a complex Lie group), see \cite[Thm. A and B]{dumas2023regularity}. Corollary \ref{coro:cpt} gives a partial result towards this conjecture : the optimal regularity cannot be greater than $\kappa(G)$. Furthermore, the conjecture is also shown in some new specific cases.

\begin{coro}\label{coro:cpt2}
    The conjecture is true for\begin{itemize}
        \item $(U,K)=(SU(p+q),S(U(p)\times U(q)))$ with $q\geq p\geq 2$, which corresponds to $G=SU(p,q)$;
        \item $(U,K)=(SO(8),U(4))$, which corresponds to $G=SO^*(8)$;
        \item $(U,K)=(SO(10),U(5))$, which corresponds to $G=SO^*(10)$;
        \item $(U,K)=(Sp(n),Sp(2)\times Sp(n-2))$ with $n\geq 4$, which corresponds to $G=Sp(2,n-2)$.
    \end{itemize}
\end{coro}
\begin{proof}
    This is a consequence of Corollary \ref{coro:cpt} and \cite[Thm. 4.9 and 4.18]{dumas2023regularity}.
\end{proof}

\subsection{Towards a lower bound}
We now know that the optimal regularity of $K$-finite matrix coefficient of $U$ on $U_r$ is at most $\kappa(G)$, with equality in many cases. We will conclude this paper by showing that equality holds in all cases, but only in some open subset of $U_r$.

In Theorem \ref{thm:regmaxcpt}, we used the method of stationary phase to obtain a sharp estimate of the spherical functions, but only on a subfamily. In order to obtain a lower bound on regularity, we need estimates for all $\mu\in \Lambda$ that are uniform. Clearly, we know that we cannot obtain sharp estimate of this kind. Indeed, as $\mu$ varies, the critical submanifold of the phase function $F_{a,\mu}$ varies in dimension.

Thus, we cannot treat $\mu$ as a parameter and use stationary phase approximation to obtain uniform estimate when $\mu$ is close to some $\mu_0$. However, in a local chart, we can see $K$ in coordinates $(x,y)$ such that $K_\mu$ is given by $\{y=0\}$, and treat $x,a,\mu$ as parameters, as done in \cite[Prop. 9.2]{DKV} and used in Theorem \ref{thm:reg}. This leads to brutal estimates, as we forget what happens in the $x$ coordinate, but sufficient for our purposes.

A second issue arises in the statement of stationary phase used before (Theorem \ref{thm:statphase}). When choosing a parameter $\mu_0\in \Lambda$, we need to consider a neighbourhood of $\mu_0$ in $\mathfrak{a}^*$. However, a necessary hypothesis is that the real part of the phase is non-positive, which is only true on $\Lambda$, thus there may not exists a neighbourhood of $\mu_0$ such that this is true.\medskip

We begin our proof with another statement of the stationary phase approximation which is well-suited to holomorphic phase functions - instead of simply $C^\infty$ as in Theorem \ref{thm:statphase}.
\begin{theorem}[{\cite[Thm 2.8 and Rem. 2.10]{sjostrand},\cite[Prop. 1.3]{bonthonneau2020fbi}}]\label{thm:statphaseholo}Let $U\times V$ be an open subset of $\C^n\times \C^k$ and $\Gamma$ a $k$-dimensional real submanifold of $V$ with boundary. Let $\Phi,u:U\times V\to \C$ be holomorphic functions. For $x\in U$, denote $\Phi_x:y\mapsto \Phi(x,y)$. Let $x_0\in U$. Assume that there exists a unique critical point $y_0$ of $\Phi_{x_0}$ in $\Gamma$, which is non-degenerate and in the interior of $\Gamma$ and that $\Phi(x_0,y_0)=0$. Furthermore, assume that $\Ima \Phi(x_0,y)\geq 0$ for any $y\in \Gamma$, and $\Ima \Phi(x_0,y)>0$ for any $y\in \partial \Gamma$.

Then, there exists an open neighbourhood $U'\subset U$ of $x_0$ such that for any $x\in U'$, the map $\Phi_x$ has a unique critical point $y_c(x)$ close to $y_0$ and $x\mapsto y_c(x)$ is holomorphic. Furthermore, there exists $C>0$ such that for any $x\in U'$, $t\geq 1$, $$\left\vert e^{-it\Phi(x,y_c(x))}\int_\Gamma e^{it\Phi(x,y)}u(x,y)dy \right\vert \leq C\Vert u\Vert_{\infty,U'\times \Gamma} t^{-k/2}.$$
\end{theorem}

If we want to estimate the integral, we need to understand the term $e^{-it\Phi(x,y_c(x))}$. It is clear that this term is bounded below as $t\to +\infty$ except when $\Ima \Phi(x,y_c(x))<0$. But since $y_c(x)$ is not necessarily real, even for $x$ real, it is \textit{a priori} not enough to understand $\Ima \Phi(x,y)$ on real points $y$.

\begin{lem}[{\cite[Lem. 1.16]{bonthonneau2020fbi}}]\label{lem:imaginarypart}
    Let $\Phi$ be as in Theorem \ref{thm:statphaseholo} and $y_c:U'\to V$ given by the theorem. Assume that $(x_0,y_0)\in \R^n\times \R^n$ and $\Ima \Hess_{\Phi_{x_0}}(y_0)$ is positive. Then there exists an open neighbourhood $U''\subset U'$ of $x_0$ such that for any $x\in U''\cap \R^n$ with $\Ima \Phi(x,y)\geq 0$ for all $y\in V\cap \R^k$, then $$\Ima \Phi(x,y_c(x))\geq 0.$$
\end{lem}
\begin{remark}
    This lemma is stated with the hypothesis that $\Ima \Phi(x,y)\geq 0$ for \textbf{all} $x,y$ real - and thus of course, $\Ima \Phi(x,y_c(x))\geq 0$ for all $x$ real. However it is clear in the proof that we can remove this assumption by adding positivity of the Hessian matrix, and we get the result only for $x$ real where the assumption holds. This is exactly what will allow us to take care of parameters $\mu_0$ on the boundary of $\Tilde{\Lambda}$.
\end{remark}

We now come back to the Lie group setting, and we will use all the notations introduced in Section \ref{sec:cptupper}. Notice that since the subgroup $K$ is compact, there exists a relatively compact open subset $\Tilde{K}$ of $K_\C$ such that $K\subset \Tilde{K}$. By compactness, up to shrinking the open subset $V$ given by Lemma \ref{lem:extan}, we may assume it is invariant by conjugation by elements of $\Tilde{K}$. Thus, the element $F_{g,\mu}(k)=\mu(H(k^{-1}gk))\in \C$ is well-defined for any $g\in V$, $k\in \Tilde{K}$, $\mu\in \mathfrak{a}_\C^*$ and is holomorphic in all variables.

Recall also that there are fundamental weights $\mu_1,\cdots,\mu_\ell\in \mathfrak{a}^*$ such that $$\Lambda = \left\lbrace\sum_{i=1}^\ell m_i\mu_i \vert m_i\in\N\right\rbrace$$and that for any $u\in U\cap V$, $\mu\in \Lambda$, $$\re \mu(H(u))\leq 0.$$Thus, defining the cone $$\Tilde{\Lambda}=\left\lbrace \sum_{i=1}^\ell t_i\mu_i \vert t_i\in \R_+\right\rbrace,$$it is clear that for any $u\in U\cap V$, $\mu\in \Tilde{\Lambda}$, $$\re \mu(H(u))\leq 0.$$
Furthermore, the function $\psi_\mu:g\mapsto \int_K e^{\mu(H(k^{-1}gk))}dk$ is also well-defined for any $\mu\in \mathfrak{a}_\C^*$.

\begin{theorem}\label{thm:regloccpt}
    For $\mu\in \mathfrak{a}_\C^*$, let $\Psi_\mu=\psi_\mu \circ \exp\vert_{iQ}$. Let $r=\lfloor \kappa(G)\rfloor$ and $\delta=\kappa(G)-r$. Let $Q_0=\{Y\in Q \vert \exp(iY)\in V\}$. Then the family $\left(\Psi_\mu\right)_{\mu\in \Lambda}$ is bounded in $C^{(r,\delta)}(iQ_0)$.
\end{theorem}
\begin{proof}
    First consider $\Tilde{Q}=\{Y\in \mathfrak{a}_\C \vert \exp(Y)\in V\}$. Let $\mu\in \mathfrak{a}_\C^*$ and $Y\in \Tilde{Q}$. Then the phase function $F_{\exp Y,\mu}:k\mapsto\mu(H(k^{-1}\exp(Y)k))$ is well-defined on $\Tilde{K}$ and holomorphic in all variables. Let $t\in\R$, consider $\Psi_{t\mu}(Y)=\int_K e^{tF_{a,\mu}(k)}dk$. Let $s\in \N$ and $X\in \mathfrak{a}^r_\C$. For $0\leq j\leq s$, we can consider the functions $$g_j(\mu,Y,X):k\in \Tilde{K}\mapsto \frac{1}{j!}\frac{d^j}{dt^j}\left.\left(\left(D^s(e^{tF_{a,\mu}(k)})(Y)(X)\right)e^{-tF_{a,\mu}(k)}\right)\right\vert_{t=0}$$where $D$ is the differential operator with respect to the variable $Y$. The functions $g_j$ are holomorphic in all variables. As in Theorem \ref{thm:reg}, we get that \begin{equation}\label{eq:derivPhicpt}D^s\Psi_{t\mu}(Y)(X)= \sum_{j=0}^s t^j \int_K e^{tF_{a,\mu}(k)}g_j(\mu,Y,X)(k)dk.\end{equation}
    Let $S=\{\mu\in \Tilde{\Lambda} \vert \, \Vert \mu\Vert=\min \Vert \mu_i\Vert\}$ which is a compact set. For now, fix $\mu_0\in S$, $Y_0\in Q_0$, $a_0=\exp(iY_0)$. Let $\mathcal{C}=\mathcal{C}_{\mu_0}$ be the critical set of $F_{a_0,\mu_0}$ computed in Lemma \ref{lem:crit}. The critical set $\mathcal{\mu}$ varies with $\mu$, but it depends on $K_\mu$ hence on the set of roots orthogonal to $\mu$. We can choose a neighbourhood $V_{\mu_0}$ of $\mu_0$ in $\mathfrak{a}_\C^*$ such that for any $\mu\in V_{\mu_0}\cap \mathfrak{a}$, $\mathcal{C}_\mu\subset \mathcal{C}$.

    Consider a tubular neighbourhood of $\mathcal{C}$, that is to a say a vector bundle $\pi:E\to \mathcal{C}$ together with a map $J:E\to K$ such that if $0_E$ denote the zero section of the vector bundle, $J(0_E(x))=x$ and $J$ is a diffeomorphism from an open neighbourhood $\Omega$ of $0_E(\mathcal{C})$ to an open neighbourhood of $\mathcal{C}$ (we refer to \cite{lee2003introduction} for more details on the construction of such a bundle). 
    For any $k_0\in \mathcal{C}$, consider an open neighbourhood $U_{1,k_0}$ which is a local trivialization of the bundle. Then $U_{2,k_0}=J(\Omega \cap \pi^{-1}(U_{1,k_0}))$ is a neighbourhood of $k_0$ in $K$. Notice that if $x\in U_{2,k}\cap U_{2,k'}$, then $\pi(J^{-1}(x))\in U_{1,k}\cap U_{1,k'}$. Let $V_{k_0}$ be a neighbourhood of $k_0$ in $\Tilde{K}$ containing $U_{2,k_0}$, which we may assume up to reducing $U_{2,k_0}$ to be the domain of an analytic chart $$\fonction{H_{k_0}}{V_{k_0}}{\C^{\dim(\mathcal{C})}\times \C^{n(\mu_0)}}{k}{(z_1,z_2)}$$ defined in such a way that $H_{k_0}(k_0)=(0,0)$, $H_{k_0}(U_{2,k_0})=H_{k_0}(V_{k_0}\cap K)=H(V_{k_0})\cap (\R^{\dim(\mathcal{C})}\times \R^{n(\mu_0)})$ and $\mathcal{C}$ is given is those local coordinates by $H_{k_0}(U_{2,k_0})\cap\{z_2=0\}$. Indeed, by the computations in Lemma \ref{lem:crit}, there exists $w\in W$ such that $k_0\subset wK_{\mu_0}$, and we may assume that $U_{1,k_0}\cap \mathcal{C}\subset wK_\mu$. Then set $O=\exp_K^{-1}(w^{-1}U_{2,k_0})$, $O_\C=O+iO$ open subset of $\mathfrak{k}_\C$ and $V_{k_0}=w\exp_{K_\C}(O_\C)$ - up to shrinking at each step so that the exponential map is a diffeomorphism in the neighbourhoods considered. Then if we consider a decomposition $\mathfrak{k}=\mathfrak{k}_\mu\oplus E$, the chart $H_{k_0}(k)=\exp_{K_\C}^{-1}(w^{-1}k)$ satisfies the requirements.
    
    Let $x=(Y,\mu,z_1)\in V_1=\Tilde{Q}\times V_{\mu_0}\times \operatorname{pr}_1(H_{k_0}(V_{k_0}))$, $y=z_2\in V_2=\operatorname{pr}_2(H_{k_0}(V_{k_0}))$. We will treat $x$ as a parameter and apply the stationary phase only in the $y$ coordinate. Let $x_0=(Y_0,\mu_0,0)$ and $\Phi(x,y)=-i(F_{\exp(iY),\mu}(H^{-1}_{k_0}(z_1,z_2))-F_{a_0,\mu_0}(k_0))$. Then $\Phi_{x_0}:y\mapsto \Phi(x_0,y)$ has a unique critical point $y_0=0$ in $V_2$, which is non-degenerate. Since $\mu_0\in \Tilde{\Lambda}$ and $\re F_{a_0,\mu_0}(k_0)=0$, we also have $\Ima\Phi(x_0,y)\geq 0$ for any $y\in V_2$. Furthermore, by \eqref{eq:hessvp}, the imaginary part of the Hessian matrix of $\Phi_{x_0}$ at $y_0$ is positive-definite. By Taylor's formula, this implies that $\Ima \Phi(x_0,y)>0$ for $y$ real and close to $y_0$. Thus, we can chose a small ball $\Gamma_{k_0}\subset V_2\cap \R^{n(\mu_0)}$ centered around $y_0$ such that $\Ima \Phi(x_0,y)>0$ for $y\in \partial \Gamma_{k_0}$. Thus, we can apply Theorem \ref{thm:statphaseholo} and Lemma \ref{lem:imaginarypart}. Thus, there is an open neighbourhood $V'_1\subset V_1$ of $x_0$, which we may assume to be of the form $V'_{Y_0}(k_0)\times V'_{\mu_0}(k_0)\times V'_{0}(k_0)$, such that for any $x\in V'_1$, the map $\Phi_x$ has a unique critical point $y_c(x)$. Up to shrinking, we may assume that $V'_{0}(k_0)$ is a ball centered at $0$, whose closure is contained in $V_{k_0}$. For $x=(Y,\mu,z_1)$ real and such that $\mu\in \Tilde{\Lambda}$, Lemma \ref{lem:imaginarypart} ensures that $\Ima \phi(x,y_c(x))\geq 0$. Finally, there is a constant $C_{Y_0,\mu_0,k_0}>0$ such that for any holomorphic map $g$ on $V_1\times V_2$, for any $t\geq 1$ and $x=(Y,\mu,z_1)\in V_1'$ real with $\mu\in \Tilde{\Lambda}$, \begin{equation}
        \label{eq:estimatechart} \left\vert \int_{\Gamma_{k_0}} e^{it\phi(x,y)}g(x,y)dx\right\vert \leq C_{Y_0,\mu_0,k_0}\Vert g\Vert_\infty t^{-n(\mu_0)/2}\leq C_{Y_0,\mu_0,k_0}\Vert g\Vert_\infty t^{-\kappa(G)}.
    \end{equation}For the last inequality, we used that $\kappa(G)\leq \frac{n(\mu)}{2}$ for any $\mu\in \mathfrak{a}^*$ by definition.

    Let $\Omega_{k_0}=H^{-1}_{k_0}((V_0'(k_0)\cap \R^{\dim \mathcal{C}})\times \{0\})=H_{k_0}^{-1}(V_0'(k_0)\times \mathring{\Gamma}_{k_0})\cap \mathcal{C}$. Then $\Omega_{k_0}$ is an open neighbourhood of $k_0$ in $\mathcal{C}$, contained in $U_{1,k_0}$. We will construct an "almost disjoint" cover of $\mathcal{C}$.  Since $\mathcal{C}=\bigcup_{k\in \mathcal{C}} \Omega_k$ is compact, we can extract a finite subcover $\Omega_{k_1},\cdots,\Omega_{k_p}$. Set $V_1=\Omega_{k_1}$ and inductively define $V_j=\Omega_{k_j}\cap \left(\mathcal{C}\setminus \bigcup_{i< j} \overline{V_i}\right)$. This construction yields $p$ disjoint open subsets $V_j$ of $\mathcal{C}$ such that $\mathcal{C}=\bigcup_{j=1}^p \overline{V_j}=\bigcup_{j=1}^p V_j \cup \bigcup_{j=1}^p \partial V_j$. We have that \begin{align*}
       \partial V_j  & \subset  \partial \Omega_{k_j} \cup \partial \left(\mathcal{C}\setminus \bigcup_{i< j} \overline{V_i}\right)\\
         & = \partial \Omega_{k_j} \cup \partial\bigcup_{i< j} \overline{V_i}\\
         & \subset \partial \Omega_{k_j} \cup \bigcup_{i< j}\partial \overline{V_i}\\
         & \subset \partial \Omega_{k_j} \cup \bigcup_{i< j}\partial V_i.
    \end{align*}Since $\partial V_1=\partial \Omega_1$, recursively we obtain $$\partial V_j \subset \bigcup_{i\leq j} \partial \Omega_{k_j}.$$But $\partial \Omega_{k_j}$ is contained in the embedded image of a $\dim \mathcal{C}-1$ dimensional sphere, thus it has Riemannian volume $0$. This implies that each $\partial V_j$ has volume $0$.

    Let $N_j=H_{k_j}^{-1}(H_{k_j}(\overline{V_j})\times \Gamma_{k_j})\subset U_{2,k_j}$ and $D=\bigcup_{j=1}^p N_j$. Since $\overline{V_j}$ are disjoints up to negligible set, and contained in $U_{1,k_j}$ local trivialization of the tubular neighbourhood, then $N_j$ are also disjoints up to negligible sets (for the volume measure on $K$). Indeed, $N_i\cap N_j\subset J(\Omega \cap \pi^{-1}(\partial V_j))$. Note that $\mathcal{C}\subset \mathring{D}$ and set $K'=K\setminus \mathring{D}$. Then $K'\cap D=\partial D\subset \bigcup_j H_{k_j}^{-1}(H_{k_j}(\partial\overline{V_j})\times \Gamma_{k_j})\cup H_{k_j}^{-1}(H_{k_j}(\overline{V_j})\times \partial\Gamma_{k_j})$ which has volume $0$ again. 
    
    Denote $A_{Y_0}=Q\cap\bigcap_{j=1}^p V'_{Y_0}(k_j)$ and $S_{\mu_0}=\bigcap_{j=1}^p V'_{\mu_0}(k_j)$, by finiteness of the intersection, these are still open neighbourhoods of $Y_0,\mu_0$ respectively, and we may assume, up to shrinking if necessary, that they are both bounded. Up to reducing once again $A_{Y_0}$ to a ball, it will be convenient to assume it is convex.

    By assumption, for $a=\exp(iY)$ with $Y\in Q$ and $\mu$ real in $S_{\mu_0}$, $\mathcal{C}_{\mu_0}\subset \mathcal{C}$ so $F_{a,\mu}$ has no critical point in $K'$ since $\mathcal{C}\cap K'=0$. By compactness of $K'$ and since with these parameters, $\re F_{a,\mu}(k)\leq 0$ for any $k\in K$, this implies that $\Vert T_{k}F_{\exp(iY),\mu}\Vert^2 -\re F_{a,\mu}(k)$ is bounded below by $\delta>0$ for $k\in K'$, $Y\in A_{Y_0}$ and $\mu\in  S_{\mu_0}\cap \Tilde{\Lambda}$. Thus, by \cite[Thm. 7.7.1]{hormander1983analysis}, for any $n\in \N$, there exists a semi-norm $\nu_{n,Y_0,\mu_0}$ on $C^\infty(K')$ such that for any $g\in C^\infty(K')$, $Y\in A_{Y_0}$, $\mu\in  S_{\mu_0}\cap \Tilde{\Lambda}$ and $t\geq 1$, \begin{equation}
        \label{eq:estimatenocrit} \left\vert \int_{K'} e^{tF_{\exp(iY),\mu}(k)}g(k)dk\right\vert \leq \nu_{n,Y_0,\mu_0}(g)t^{-n}.
    \end{equation}
    Now, as in Theorem \ref{thm:regmaxcpt}, let $d_0k$ denote the volume measure on $K$ associated to the invariant Riemannian metric $h$ induced by the inner product $-\langle \cdot ,\cdot\rangle$ and $\operatorname{Vol}(K)=\int_K d_0K$. Then the Haar measure $dk$ is $\frac{1}{\operatorname{Vol}(K)}d_0k$. Let also $G(k)=\det\left(h_k(\partial_i,\partial_j)\right)$. Denote also $$\Tilde{g}_{j,m}(\mu,Y,X)(z_1,z_2)=g_j(\mu,Y,X)(k)\sqrt{G(k)}$$where $(z_1,z_2)=H_{k_m}(k)$. Since by construction, the set of points counted several times is negligible, we we can write that for any $1\leq j\leq s$, \begin{multline*} \left\vert \int_K e^{tF_{\exp(iY),\mu}(k)}g_j(\mu,Y,X)(k)dk\right\vert \\ \begin{aligned}
        & \leq \left\vert \int_{K'} e^{tF_{\exp(iY),\mu}(k)}g_j(\mu,Y,X)(k)dk\right\vert\\
        & \phantom{\leq }+\sum_{m=1}^p \left\vert \int_{N_m} e^{tF_{\exp(iY),\mu}(k)}g_j(\mu,Y,X)(k)dk\right\vert \\
        & \leq \left\vert \int_K' e^{tF_{\exp(iY),\mu}(k)}g_j(\mu,Y,X)(k)u_0(k)dk\right\vert\\
        & \phantom{\leq }+\frac{1}{\operatorname{Vol}(K)}\sum_{m=1}^p \left\vert \int_{H_{k_m}(\overline{V_m})}\int_{\Gamma_{k_m}} e^{it\phi((Y,\mu,z_1),z_2)}\Tilde{g}_{j,m}(\mu,Y,X)(z_1,z_2)dz_2dz_1\right\vert \\
    \end{aligned}         
    \end{multline*}
We can now apply \eqref{eq:estimatechart}, and \eqref{eq:estimatenocrit} with $n=\lceil \kappa(G)\rceil$ to get that for any $Y\in A_{Y_0}$, $\mu\in S_{\mu_0}\cap \Tilde{\Lambda}$, $X\in \mathfrak{a}^s$ and $t\geq 1$, \begin{align*} \left\vert \int_K e^{tF_{\exp(iY),\mu}(k)}g_j(\mu,Y,X)(k)dk\right\vert 
        & \leq \nu_{\lceil \kappa(G)\rceil,Y_0,\mu_0}(g_j(\mu,Y,X)u_0)t^{-\lceil \kappa(G)\rceil}\\&\phantom{\leq}+\sum_{m=1}^p \frac{C_{Y_0,\mu_0,k_i}\operatorname{Vol}H_{k_m}(\overline{V_m})}{\operatorname{Vol}(K)} \Vert \Tilde{g}_{j,m}(\mu,Y,X)\Vert_\infty t^{-\kappa(G)}
\end{align*}
Now since the functions $g_j$ are smooth in all variables, they are bounded in $C^\infty(K)$ when $\mu,Y,X$ remain bounded, thus there exists a constant $D_{j,Y_0,\mu_0}>0$ such that for any $Y\in A_{Y_0}$, $\mu\in S_{\mu_0}\cap \Tilde{\Lambda}$, $X\in \mathfrak{a}^s$ with $\Vert X_i\Vert=1$ and any $t\geq 1$, \begin{equation}
    \label{eq:majcptint} \left\vert \int_K e^{tF_{\exp(iY),\mu}(k)}g_j(\mu,Y,X)(k)dk\right\vert \leq D_{j,Y_0,\mu_0}t^{-\kappa(G)}.
\end{equation}
Since $S$ is compact and $S\subset \bigcup_{\mu_0\in S} S_{\mu_0}$, we can extract a finite cover $S_{\mu_1},\cdots,S_{\mu_q}$. Set $$M_{Y_0,s}=\underset{1\leq k\leq q}{\max}\sum_{j=1}^s D_{j,Y_0,\mu_k}.$$
Combining \eqref{eq:majcptint} with \eqref{eq:derivPhicpt}, for any $Y\in A_{Y_0}$, $\mu\in S$ and any $t\geq 1$, \begin{equation}
    \label{eq:majDPHI} \Vert D^s\Psi_{t\mu}(Y)\Vert = \underset{\Vert X_i\Vert=1}{\sup} \vert D^s\Psi_{t\mu}(Y)(X)\vert \leq M_{Y_0,s}t^{s-\kappa(G)}.
\end{equation}
Note that $\Lambda\setminus\{0\}\subset \R_{\geq 1}S$, so the previous equation immediately implies that for any $\mu\in \Lambda\setminus\{0\}$ and $s\leq r$, \begin{equation}
    \label{eq:majDPHI2} \Vert D^s\Psi_{t\mu}(Y)\Vert \leq M_{Y_0,s}.
\end{equation}
Thus if $\kappa(G)$ is an integer, the proof is complete.

Otherwise, $\delta=\kappa(G)-r=\frac{1}{2}$. Then applying \eqref{eq:majDPHI} with $s=r$ and triangular inequality gives that for any $Y,Y'\in A_{Y_0}$, $\mu\in S$, $t\geq 1$, $$\Vert D^r\Psi_{t\mu}(Y)-D^r\Psi_{t\mu}(Y')\Vert\leq 2M_{Y_0,r}t^{-1/2}.$$On the other hand, using \eqref{eq:majDPHI} with $s=r+1$ and the mean value theorem gives that for any $Y,Y'\in A_{Y_0}$, $\mu\in S$, $t\geq 1$, $$\Vert D^{r}\Psi_{t\mu}(Y)-D^r\Psi_{t\mu}(Y')\Vert\leq \underset{Z\in A_{Y_0}}{\sup}\Vert D^{r+1}\Psi_{t\mu}(Z)\Vert \Vert Y-Y'\Vert\leq M_{Y_0,r+1}t^{1/2}\Vert Y-Y'\Vert.$$
Combining both estimates, we get that for any $Y,Y'\in A_{Y_0},\mu\in \Lambda\setminus \{0\}$, \begin{equation}
    \label{eq:finalestimate} \Vert D^{r}\Psi_{t\mu}(Y)-D^r\Psi_{t\mu}(Y')\Vert \leq \left(2M_{Y_0,r}M_{Y_0,r+1}\right)^{1/2} \Vert Y-Y'\Vert
\end{equation}which completes the proof when $\kappa(G)$ is not an integer.
\end{proof}
    
\begin{coro}\label{coro:regvoisinage} Consider $(U,K)$ as before. Let $r=\lfloor \kappa(G)\rfloor$ and $\delta=\kappa(G)-r$. Let $Q_0=\{Y\in Q \vert \exp(iY)\in V\}$ and $U_0=K\exp(iQ_0)K\subset U_r$. Then any $K$-finite matrix coefficient of a unitary representation of $U$ is in $C^{(r,\delta)}(V)$.
\end{coro}
\begin{proof}
    By Theorem \ref{thm:regloccpt}, Proposition \ref{prop:kakcpt} and Lemma \ref{lem:precomposition}, the family of spherical functions of $(U,K)$ is bounded in $C^{(r,\delta)}(U_0)$. By Lemma \ref{lem:lienspheriquekbiinv}, it follows that any $K$-bi-invariant matrix coefficient of a unitary representation of $U$ is in $C^{(r,\delta)}(U_0)$. Finally, \cite[Thm. 5.2]{dumas2023regularity} allows to extend this regularity to $K$-finite coefficients.
\end{proof}

\begin{remark}
    Combining Corollaries \ref{coro:cpt} and \ref{coro:regvoisinage}, we get that Conjecture \ref{conj} is true for any compact symmetric pair $(U,K)$, but only in some open subset $U_0$ and not all of $U_r$.

    The same proof cannot extend this regularity to all of $U_r$. It is clear that for any $g\in K_\C A_\C N_\C$ - in particular for any $a\in U_r\cap A_\C$ - we can consider an analytic extension of $H$ in a neighbourhood of $g$. However, since $g$ is not a fixed point of the action by conjugation it cannot be chosen $K$-invariant as in Lemma \ref{lem:extan}, thus we cannot get the integral expression of spherical functions around $g$ to work with.

    However Clerc gave in \cite{clerc} a multivalued analytic extension of $H$ to all of $K_\C A_\C N_\C$, as well as an integral formula of spherical functions of $(U,K)$ very similar to Lemma \ref{lem:spheriqueext}. But the domain of integration in the expression of $\varphi_\mu(g)$ is now an open subset $K_g$ of $K$. Thus, the lack of compactness does not allow for a similar proof.

    When $\mu$ is regular - meaning that $\langle\mu ,\alpha\rangle \neq 0$ for any root $\alpha$ - Clerc managed to reduce this integral to a compact subset and get precise estimate of spherical functions. However, this is not sufficient for our purposes, and this does not work when $\mu$ is not regular. 
\end{remark}

\nocite{*}
\addcontentsline{toc}{section}{Bibliography}
\bibliographystyle{alpha}
\bibliography{Ref}

\begin{thebibliography}{dLMdlS16}

\bibitem[BJ20]{bonthonneau2020fbi}
Y.G. Bonthonneau and M.~Jézéquel.
\newblock {FBI} transform in {G}evrey classes and {A}nosov flows, 2020.

\bibitem[Bou07]{bourbaki2007groupes}
N.~Bourbaki.
\newblock {\em Groupes et alg{\`e}bres de Lie: Chapitres 4, 5 et 6}.
\newblock Bourbaki, Nicolas. Springer Berlin Heidelberg, 2007.

\bibitem[Cha74]{Chazarain1974}
J.~Chazarain.
\newblock Formule de poisson pour les variétés riemanniennes.
\newblock {\em Inventiones mathematicae}, 24:65--82, 1974.

\bibitem[Cle76]{Clers1976}
J.-L. Clerc.
\newblock Une formule asymptotique du type {M}ehler-{H}eine pour les zonales d'un espace riemannien symétrique.
\newblock {\em Studia Mathematica}, 57(1):27--32, 1976.

\bibitem[Cle88]{clerc}
J.-L. Clerc.
\newblock Fonctions sphériques des espaces symétriques compacts.
\newblock {\em Trans. Amer. Math. Soc.}, 306(1):421--431, 1988.

\bibitem[DKV83]{DKV}
J.~J. Duistermaat, J.~A.~C. Kolk, and V.~S. Varadarajan.
\newblock Functions, flows and oscillatory integrals on flag manifolds and conjugacy classes in real semisimple {Lie} groups.
\newblock {\em Compositio Mathematica}, 49(3):309--398, 1983.

\bibitem[dLMdlS16]{dLMdlS}
T.~de~Laat, M.~Mimura, and M.~de~la Salle.
\newblock On strong property {(T)} and fixed point properties for {Lie} groups.
\newblock {\em Annales de l'Institut Fourier}, 66(5):1859--1893, 2016.

\bibitem[Dum24]{dumas2023regularity}
G.~Dumas.
\newblock Regularity of matrix coefficients of a compact symmetric pair of {L}ie groups.
\newblock {\em Trans. Amer. Math. Soc.}, 377(10):7421--7474, 2024.

\bibitem[FJK79]{flensted-jensen}
M.~Flensted-Jensen and T.H. Koornwinder.
\newblock Positive definite spherical functions on a non-compact, rank one symmetric space.
\newblock In Pierre Eymard, Reiji Takahashi, Jacques Faraut, and G{\'e}rard Schiffmann, editors, {\em Analyse Harmonique sur les Groupes de Lie II}, pages 249--282, Berlin, Heidelberg, 1979. Springer Berlin Heidelberg.

\bibitem[HC53]{harishchandra}
Harish-Chandra.
\newblock Representations of a semisimple {L}ie group on a {B}anach space. {I}.
\newblock {\em Transactions of the American Mathematical Society}, 75(2):185--243, 1953.

\bibitem[Hel79]{helgason1979differential}
S.~Helgason.
\newblock {\em Differential Geometry, Lie Groups, and Symmetric Spaces}.
\newblock ISSN. Elsevier Science, 1979.

\bibitem[Hel00]{helgason2000groups}
S.~Helgason.
\newblock {\em Groups and Geometric Analysis: Integral Geometry, Invariant Differential Operators, and Spherical Functions}.
\newblock Mathematical surveys and monographs. American Mathematical Society, 2000.

\bibitem[H{\"o}r83]{hormander1983analysis}
L.~H{\"o}rmander.
\newblock {\em The Analysis of Linear Partial Differential Operators I: Distribution theory and Fourier analysis}.
\newblock A series of comprehensive studies in mathematics. Springer-Verlag, 1983.

\bibitem[Kir76]{kirillov1976elements}
A.A. Kirillov.
\newblock {\em Elements of the Theory of Representations}.
\newblock Grundlehren der mathematischen Wissenschaften. SPRINGER-VERLAG *, 1976.

\bibitem[Kna01]{knapp2001representation}
A.W. Knapp.
\newblock {\em Representation Theory of Semisimple Groups: An Overview Based on Examples}.
\newblock Princeton landmarks in mathematics and physics. Princeton University Press, 2001.

\bibitem[Kna02]{knapp2002lie}
A.W. Knapp.
\newblock {\em Lie Groups Beyond an Introduction}.
\newblock Progress in Mathematics. Birkh{\"a}user Boston, 2002.

\bibitem[Kos69]{kostant}
B.~Kostant.
\newblock {On the existence and irreducibility of certain series of representations}.
\newblock {\em Bulletin of the American Mathematical Society}, 75(4):627 -- 642, 1969.

\bibitem[Laf08]{lafforgue}
V.~Lafforgue.
\newblock {Un renforcement de la propriété (T)}.
\newblock {\em Duke Math. J.}, 143(3):559 -- 602, 2008.

\bibitem[Lee03]{lee2003introduction}
J.M. Lee.
\newblock {\em Introduction to Smooth Manifolds}.
\newblock Graduate Texts in Mathematics. Springer, 2003.

\bibitem[Lee19]{lee2019introduction}
J.M. Lee.
\newblock {\em Introduction to Riemannian Manifolds}.
\newblock Graduate Texts in Mathematics. Springer International Publishing, 2019.

\bibitem[Loo69]{loos1969symmetric2}
O.~Loos.
\newblock {\em Symmetric Spaces Vol.2 : Compact spaces and classification}.
\newblock Mathematics lecture note series. W. A. Benjamin, 1969.

\bibitem[MS75]{melin}
A.~Melin and J.~Sj{\"o}strand.
\newblock Fourier integral operators with complex-valued phase functions.
\newblock In J.~Chazarain, editor, {\em Fourier Integral Operators and Partial Differential Equations}, pages 120--223, Berlin, Heidelberg, 1975. Springer Berlin Heidelberg.

\bibitem[Sj{\"o}82]{sjostrand}
J.~Sj{\"o}strand.
\newblock Singularit\'es analytiques microlocales.
\newblock In {\em Singularit\'es analytiques microlocales - \'equation de Schr\"odinger et propagation des singularit\'es...}, number~95 in Ast\'erisque, pages iii--166. Soci\'et\'e math\'ematique de France, 1982.

\bibitem[Var13]{varadarajan2013lie}
V.S. Varadarajan.
\newblock {\em Lie Groups, Lie Algebras, and Their Representations}.
\newblock Graduate Texts in Mathematics. Springer New York, 2013.

\bibitem[vD09]{Dijk+2009}
G.~van Dijk.
\newblock {\em Introduction to Harmonic Analysis and Generalized Gelfand Pairs}.
\newblock De Gruyter, 2009.

\bibitem[Vre76]{vretare_orth}
L.~Vretare.
\newblock Elementary spherical functions on symmetric spaces.
\newblock {\em Mathematica Scandinavica}, 39(2):343--358, 1976.

\end{thebibliography}

\appendix

\section{Computation of \texorpdfstring{$\kappa(G)$}{K(G)}}
\label{sec:calculkappa}
If $G$ is semisimple with finite center, then $G/Z(G)= \prod_{i=1}^n G_i$ with $G_i$ simple. Then, $\kappa(G)=\underset{i\textrm{ s.t. }G_i\textrm{ not compact}}{\min} \kappa(G_i)$. Indeed, the restricted root system of $G$ is the direct sum of the restricted root system of the $G_i$'s, and the restricted root system of a compact group is trivial (because it is its own maximal compact subgroup). Thus, it suffices to compute $\kappa$ for noncompact simple Lie groups.\smallskip

Table \ref{tab:kappacomplex} deals with complex simple Lie groups (viewed as real Lie groups) and their compact real form $U$. In this case, all multiplicities are $2$. In Table \ref{tab:kappareel}, $M$ is the compact symmetric space dual to $G/K$, that is to say $U/K$ where $U$ is a compact real form of $G$. The computation uses the classification and multiplicities for such symmetric spaces given in \cite[Ch. VII]{loos1969symmetric2}.

\begin{table}[ht]
\caption{Values of $\kappa(G)$ for complex simple Lie groups.}
\label{tab:kappacomplex}
\renewcommand{\arraystretch}{1.2}
\centering
\begin{tabular}{|c|c|c|c|c|}
\hline
$G$ &$U$                 & $\dim G$  & $\rank G$ & $\kappa(G)$    \\ \hline
$SL(n,\C)$ & $SU(n), n\geq 2$    & $2(n^2-1)$   & $n-1$     & $n-1$  \\ \hline
$SO(2n+1,\C)$&$SO(2n+1), n\geq 1$ & $2n(2n+1)$ & $n$       & $2n-1$ \\ \hline
$Sp(2n,\C)$&$Sp(n), n\geq 1$    & $2n(2n+1)$ & $n$       & $2n-1$ \\ \hline
$SO(2n,\C)$&$SO(2n), n\geq 2$   & $2n(2n-1)$ & $n$       & \begin{tabular}{@{}c@{}} $1$ if $n=2$\\$3$ if $n=3$ \\ $2n-2$ else\end{tabular} \\ \hline
$(G_2)_\C$& $G_2$               & $28$      & $2$       & $5$    \\ \hline
$(F_4)_\C$&$F_4$               & $104$      & $4$       & $15$   \\ \hline
$(E_6)_\C$&$E_6$               & $156$      & $6$       & $16$   \\ \hline
$(E_7)_\C$&$E_7$               & $266$     & $7$       & $27$   \\ \hline
$(E_8)_\C$&$E_8$               & $496$     & $8$       & $57$   \\ \hline
\end{tabular}

\end{table}

\begin{table}[ht]
\caption{Values of $\kappa(G)$ for real simple Lie groups.}
\label{tab:kappareel}
\renewcommand{\arraystretch}{1.3}
\centering
\centerline{\begin{tabular}{|cc|c|c|c|}
\hline
\multicolumn{2}{|c|}{$M$}                                                                  & $G$        & $\rank G$ & $\kappa(G)$ \\ \hline
\multicolumn{1}{|c|}{$AI$}                    & $SU(n)/SO(n),n\geq 2$                              & $SL(n,\R)$ & $n-1$              &    $\frac{n-1}{2}$ \\ \hline
\multicolumn{1}{|c|}{$AII$}                   & $SU(2n)/Sp(n),n\geq 2$                             & $SU^*(2n)$          & $n-1$              &  $2(n-1)$   \\ \hline
\multicolumn{1}{|c|}{$AIII$}                  & $SU(p+q)/S(U(p)\times U(q)), p+q\geq 3$               & $SU(p,q)$                  & $\min(p,q)$        &  \begin{tabular}{@{}c@{}}$2$ if $p=q=2$ \\ $p+q-\frac{3}{2}$ else\end{tabular}   \\ \hline
\multicolumn{1}{|c|}{$BDI$}                   & $SO(p+q)/SO(p)\times SO(q),p+q\geq 3$ & $SO_0(p,q)$                   & $\min(p,q)$        &   \begin{tabular}{@{}c@{}}$\frac{1}{2}$ if $p=q=2$ \\ $\frac{3}{2}$ if $p=q=3$\\$\frac{p+q}{2}-1$ else\end{tabular}  \\ \hline
\multicolumn{1}{|c|}{$CI$}                    & $Sp(n)/U(n),n\geq 1$                               & $Sp(2n,\R)$               & $n$                &   $n-\frac{1}{2}$  \\ \hline
\multicolumn{1}{|c|}{$CII$}                   & $Sp(p+q)/Sp(p)\times Sp(q),p+q\geq 2$                & $Sp(p,q)$                  & $\min(p,q)$        &   \begin{tabular}{@{}c@{}}$5$ if $p=q=2$ \\ $2(p+q)-\frac{5}{2}$ else\end{tabular}   \\ \hline
\multicolumn{1}{|c|}{\multirow{2}{*}{$DIII$}} & $SO(4n)/U(2n), n\geq 1$                             & $SO^*(4n)$             & $n$                &  \begin{tabular}{@{}c@{}}$n\left(n-\frac{1}{2}\right)$ if $n\leq 3$ \\ $4n-\frac{7}{2}$ if $n>3$\end{tabular}   \\ \cline{2-5} 
\multicolumn{1}{|c|}{}                        & $SO(4n+2)/U(2n+1),n\geq 1$                         & $SO^*(4n+2)$             & $n$                &   $4n-\frac{3}{2}$  \\ \hline
\multicolumn{2}{|c|}{$EI$}                                                                 &               $E_{6(6)}$    & $6$                &  $8$   \\ \hline
\multicolumn{2}{|c|}{$EII$}                                                                &               $E_{6(2)}$    & $4$                &   $\frac{21}{2}$  \\ \hline
\multicolumn{2}{|c|}{$EIII$}                                                               &                $E_{6(-14)}$   & $2$                & $\frac{21}{2}$    \\ \hline
\multicolumn{2}{|c|}{$EIV$}                                                                &               $E_{6(-26)}$  & $2$                &  $8$   \\ \hline
\multicolumn{2}{|c|}{$EV$}                                                                 &                 $E_{7(7)}$   & $7$                &   $\frac{27}{2}$  \\ \hline
\multicolumn{2}{|c|}{$EVI$}                                                                &                 $E_{7(-5)}$  & $4$                &   $\frac{33}{2}$  \\ \hline
\multicolumn{2}{|c|}{$EVII$}                                                               &                $E_{7(-24)}$    & $3$                &  $\frac{27}{2}$   \\ \hline
\multicolumn{2}{|c|}{$EVIII$}                                                              &                $E_{8(8)}$  & $8$                &  $\frac{57}{2}$   \\ \hline
\multicolumn{2}{|c|}{$EIX$}                                                                &                $E_{8(-24)}$   & $4$                &   $\frac{57}{2}$  \\ \hline
\multicolumn{2}{|c|}{$FI$}                                                                 &                $F_{4(4)}$  & $4$                &   $\frac{15}{2}$ \\ \hline
\multicolumn{2}{|c|}{$FII$}                                                                &             $F_{4(-20)}$      & $1$                &   $\frac{15}{2}$  \\ \hline
\multicolumn{2}{|c|}{$G$}                                                                  &                  $G_{2(2)}$   & $2$                &   $\frac{5}{2}$  \\ \hline
\end{tabular}}

\end{table}

\end{document}